 \documentclass[11pt]{article}

\usepackage{palatino}
\usepackage[mathcal]{euler}

\usepackage{amsmath,amsthm,amsfonts,amssymb,latexsym,amscd,enumerate,xypic}

\theoremstyle{plain}
    \newtheorem{theorem}{Theorem}[section]
    \newtheorem{lemma}[theorem]{Lemma}
    
    \newtheorem{proposition}[theorem]{Proposition}

    \newtheorem{assumption}[theorem]{Assumption}
 
 \theoremstyle{definition}
    \newtheorem{definition}[theorem]{Definition}
    
    \newtheorem{example}[theorem]{Example}

    \newtheorem{remark}[theorem]{Remark}

\theoremstyle{remark}

\numberwithin{equation}{section}

\DeclareMathOperator{\Ad}{Ad}

\DeclareMathOperator{\ind}{index}

\DeclareMathOperator{\End}{End}
\DeclareMathOperator{\Hom}{Hom}
\DeclareMathOperator{\sgn}{sgn}

\DeclareMathOperator{\Spin}{Spin}

\DeclareMathOperator{\U}{U}

 \DeclareMathOperator{\Zeroes}{Zeroes}
  \DeclareMathOperator{\Span}{span}
    \DeclareMathOperator{\Br}{Br}

         \DeclareMathOperator{\supp}{supp}

      \DeclareMathOperator{\DInd}{D-Ind}
\DeclareMathOperator{\Aut}{Aut}

\begin{document}

\newcommand{\Spinc}{\Spin^c}

    \newcommand{\R}{\mathbb{R}}
    \newcommand{\C}{\mathbb{C}} 
    \newcommand{\N}{\mathbb{N}}
    \newcommand{\Z}{\mathbb{Z}} 
    \newcommand{\Q}{\mathbb{Q}}
    \newcommand{\bK}{\mathbb{K}}

\newcommand{\g}{\mathfrak{g}}
\newcommand{\h}{\mathfrak{h}}
\newcommand{\p}{\mathfrak{p}}
\newcommand{\kg}{\mathfrak{g}} 
\newcommand{\kt}{\mathfrak{t}}
\newcommand{\kA}{\mathfrak{A}}
\newcommand{\XX}{\mathfrak{X}}
\newcommand{\kh}{\mathfrak{h}} 
\newcommand{\kp}{\mathfrak{p}}
\newcommand{\kk}{\mathfrak{k}}

\newcommand{\cE}{\mathcal{E}}
\newcommand{\cA}{\mathcal{A}}
\newcommand{\calL}{\mathcal{L}}
\newcommand{\calH}{\mathcal{H}}
\newcommand{\cO}{\mathcal{O}}
\newcommand{\cB}{\mathcal{B}}
\newcommand{\cK}{\mathcal{K}}
\newcommand{\cP}{\mathcal{P}}
\newcommand{\calD}{\mathcal{D}}
\newcommand{\cF}{\mathcal{F}}
\newcommand{\cX}{\mathcal{X}}
\newcommand{\cS}{\mathcal{S}}
\newcommand{\cU}{\mathcal{U}}

\newcommand{\Sj}{ \sum_{j = 1}^{\dim G}}
\newcommand{\Sk}{ \sum_{k = 1}^{\dim M}}
\newcommand{\ii}{\sqrt{-1}}

\newcommand{\ddt}{\left. \frac{d}{dt}\right|_{t=0}}

\newcommand{\cSM}{\cS}
\newcommand{\PM}{P}
\newcommand{\DM}{D}
\newcommand{\LM}{L}
\newcommand{\vM}{v}

\newcommand{\Wedge}{\Lambda}

\newcommand{\specialin}{\hspace{-1mm} \in \hspace{1mm} }

\newcommand{\beq}[1]{\begin{equation} \label{#1}}
\newcommand{\eeq}{\end{equation}}
\newcommand{\bspl}{\[ \begin{split}}
\newcommand{\espl}{\end{split} \]}

\newcommand{\Utilde}{\widetilde{U}}
\newcommand{\Btilde}{\widetilde{B}}
\newcommand{\Dtilde}{\widetilde{D}}
\newcommand{\Etilde}{\widetilde{\cE}}

\newcommand{\Rhat}{\widehat{R}}

\title{An equivariant index for proper actions I}

\author{Peter Hochs\footnote{University of Adelaide, \texttt{peter.hochs@adelaide.edu.au}} \hspace{1mm}and Yanli Song\footnote{Dartmouth College, \texttt{song@math.dartmouth.edu}}}

\date{\today}

\maketitle

\begin{abstract}
Equivariant indices, taking values in group-theoretic objects,  have previously been defined in cases where either the group acting or the orbit space of the action is compact. In this paper, we define an equivariant index without assuming the group or the orbit space to be compact. This allows us to generalise an index of deformed Dirac operators, defined for compact groups by Braverman. In parts II and III of this series, we explore some properties and applications of this index.
%In the case of compact orbit spaces, we show how the index is related to the analytic assembly map from the Baum--Connes conjecture, Kasparov's generalisation of Atiyah's index of transversally elliptic operators, and an index used by Mathai and Zhang. We use the index to define a notion of $K$-homological Dirac induction, and show that, under conditions, it satisfies the quantisation commutes with reduction principle.
\end{abstract}

\tableofcontents

%%%%%%%%%
%%% Intro %%%
%%%%%%%%%

\section{Introduction}

\subsection*{Background}

Equivariant index theory has a long and successful history, with applications in various areas of geometry and representation theory. To set the stage, let $G$ be a Lie group, acting properly on a manifold $M$. Let $\cE = \cE^+ \oplus \cE^-\to M$ be a $G$-equivariant, $\Z_2$-graded, Hermitian vector bundle, and $D$ an odd, self-adjoint, $G$-equivariant elliptic differential operator on $\cE$.
In the basic form of equivariant index theory, one assumes $G$ and $M$ to be compact. Then the kernel of $D$ is finite-dimensional.
Hence one can define the equivariant index of $D$ as
\[
\ind_G D := [\ker D^+] - [\ker D^-] \quad \in R(G).
\]
Here $D^{\pm}$ is the restriction of $D$ to sections of $\cE^{\pm}$, and $R(G)$ is the representation ring of $G$, whose elements are formal differences of isomorphism classes of finite-dimensional representations.

Generalisations of equivariant index theory to noncompact manifolds or groups have been obtained in two distinct directions.
\begin{enumerate}
\item If $M$ and $G$ may be noncompact, but the orbit space $M/G$ is compact (we then call the action \emph{cocompact}), one can apply the \emph{analytic assembly map} introduced by Kasparov \cite{Kasparov83} and used in the Baum--Connes conjecture \cite{Connes94}. This has been studied very intensively in the last few decades.  Successes of this area of index theory include the description of the $K$-theory of group $C^*$-algebras as in the Baum-Connes and Connes--Kasparov conjectures, and applications to the Novikov conjecture. 
Furthermore, Kasparov \cite{Kasparov14} generalised Atiyah's index of transversally elliptic operators to the cocompact case.
 Index formulas for other indices were proved in \cite{Pflaum15, Wang14}. On homogeneous spaces, important results were obtained in \cite{Atiyah77, Connes82}. 
\item If $G$ is compact, then one can often define an equivariant index of a suitable deformation of $D$. For the \emph{trivial} group, some, but by no means all, well-known results on index theory on noncompact manifolds include the ones in \cite{Anghel93, Bruening90, Bruening92a, Bruening92b, Bunke95, Callias78, Elliott96, Gromov83, Kucerovsky01}. For nontrivial compact groups, a natural deformation of Dirac operators poses a technical challenge related to unboundedness of the anticommutator of the Dirac operator and the deformation term. This was solved by Braverman \cite{Braverman02}. The resulting index, including other, equivalent definitions, was used with great success in geometric quantisation, see e.g.\ \cite{HochsSong15, Zhang14, Paradan11}.
\end{enumerate}

The techniques used in these two cases, where $M/G$ or $G$ is compact, are very different. If $M/G$ is compact, then one can apply methods from $K$-theory and $K$-homology of $C^*$-algebras, while if $G$ is compact, then suitable deformations, or assumptions on the behaviour of the operator towards infinity, lead to indices in the completed representation ring
\[
\Rhat(G) = \Hom_{\Z}(R(G), \Z),
\]
which contains infinite direct sums of irreducible representations, with finite multiplicities. (Operator algebraic techniques are used in the treatments of Callias-type operators in \cite{Bunke95, Kucerovsky01} for the trivial group, but those techniques do not apply to the operators we are interested in.) This difference in approaches probably is an important reason why so far, no equivariant index theory has been developed that applies in cases where both $M/G$ and $G$ may be noncompact.
%Up to now, the authors are not aware of any notion of equivariant index theory in cases where both $M/G$ and $G$ may be noncompact. 
This would have the potential for applications in representation theory of noncompact Lie groups, via non-cocompact actions, for example on (co)tangent bundles to homogeneous spaces, or on coadjoint orbits of groups containing $G$.

We should point out that by an equivariant index, we mean an index taking values in an object defined purely in terms of $G$ (such as $R(G)$ or $\Rhat(G)$ if $G$ is compact). For example, the equivariant coarse index (see \cite{Higson93}, among many references), has been shown to be relevant for many problems in the noncompact setting. But because it takes values in the $K$-theory group of the equivariant Roe algebra of $M$, it is not the kind of index we are looking for here. 

Furthermore, in cases where $M/G$ and $G$ are both noncompact, index theory has been developed in terms of $G$-invariant sections \cite{Braverman14, Mathai13}. This contains information about multiplicities of the trivial representation, but in a fundamental way, the techniques used cannot be used to treat nontrivial representations.

\subsection*{The main result}

Our goal in this paper is to develop and apply equivariant index theory for proper actions by possibly noncompact groups, with possibly noncompact orbit spaces. Motivated by work by Kasparov \cite{Kasparov14} and Braverman \cite{Braverman02}, we define the notion of \emph{$G$-Fredholm} operators. For such operators, we define an equivariant index that generalises an index  of transversally elliptic operators defined by Kasparov in the cocompact case, and an index of deformed Dirac operators for actions by compact groups, developed by Braverman. See Table \ref{table intro}. 
\begin{table} \label{table intro}
\begin{tabular}{|c|c|c|}
\hline 
 & \begin{tabular}{l}$M/G$ compact,\\ $D$ transversally elliptic \end{tabular} & \begin{tabular}{l}$M/G$ noncompact, \\ $D$ a deformed Dirac operator \end{tabular}\\
\hline
$G$ compact & Atiyah, 1974 \cite{Atiyah74} & Braverman, 2002 \cite{Braverman02}\\
\hline
$G$ noncompact & Kasparov, 2015 \cite{Kasparov14} & Theorem \ref{thm def Dirac G Fred} \\
\hline
\end{tabular}
\caption{Special cases of the index}
\end{table}
The main result in this paper is that the index we introduce allows us to complete this table, by filling the bottom-right entry, see Theorem \ref{thm def Dirac G Fred}.

In the second part of this series \cite{HochsSong16b}, we study
 some properties and applications of the equivariant index of deformed Dirac operators. These include an induction property, relations with the analytic assembly map and an index used by Mathai and Zhang in \cite{Mathai10}, a notion of Dirac induction (as in the Connes--Kasparov conjecture) based on non-cocompact actions, and a \emph{quantisation commutes with reduction} property.
 
 In the third part \cite{HochsSongDS}, we consider $\Spinc$-Dirac operators. For semisimple Lie groups with discrete series representations, the equivariant index is then directly related to multiplicities of discrete series representations, in cases where the Riemannian metric has a certain product form. Furthermore, the invariant index studied in \cite{Braverman14, Mathai13} can be recovered from the equivariant index. This leads to quantisation commutes with reduction results for this invariant index and an index in terms of multiplicities of discrete series representations, and to Atiyah--Hirzebruch type vanishing results in the cocompact $\Spin$ case.

\subsection*{The $G$-index}

We now give some more technical details of the definition of the index we use. Let $K<G$ be a maximal compact subgroup. Consider the crossed product $C^*$-algebra $C_0(G/K)\rtimes G$. If $M/G$ is compact, and $D$ is transversally elliptic, then Kasparov \cite{Kasparov14} showed that $D$ defines a natural class in the $K$-homology group $KK(C_0(G/K)\rtimes G, \C)$ of $C_0(G/K)\rtimes G$. The algebra $C_0(G/K)\rtimes G$ is Morita-equivalent to the group $C^*$-algebra $C^*K$, so that this index can be viewed as an element of
\[
KK(C^*K, \C) \cong  \Rhat(K).
\]
%This completion of the representation ring contains infinite direct sums of irreducible representations, with finite multiplicities. 
Note that, even though the index can be identified with an element of $\Rhat(K)$, it depends on the action by the whole group $G$. (The identification $KK(C_0(G/K)\rtimes G, \C) \cong \Rhat(K)$ involves an induction procedure from $K$ to $G$.)

On the other hand, suppose that $D$ is a Dirac-type operator. Let $\psi\colon M\to \kg$ (with $\kg$ the Lie algebra of $G$), be an equivariant map. It induces a vector field $v^{\psi}$, which at a point $m\in M$ takes the value
\[
v^{\psi}_m := \ddt \exp(-t\psi(m))\cdot m.
\]
Then we have the \emph{deformed Dirac operator}
\[
D_{\psi} := D - \ii c(v^{\psi}).
\]
Suppose that the set of zeroes of $v^{\psi}$ is cocompact.
If $G=K$ is compact, then Braverman \cite{Braverman02} showed that such an operator has a well-defined equivariant index in $\Rhat(K)$, after rescaling the map $\psi$ by a function with suitable growth behaviour.
In this case, one has the direct equality $C_0(G/K)\rtimes G = C^*K$, and Braverman's index equals a natural class defined by $D_{\psi}$ in $KK(C^*K, \C)$.

Motivated by these two examples, we define an operator to be \emph{$G$-Fredholm} if it defines a class in $KK(C_0(G/K)\rtimes G, \C)$. This class is then its equivariant index,  special cases of  which were summarised in Table \ref{table intro}. In Theorem \ref{thm def Dirac G Fred}, we show that deformed Dirac operators are $G$-Fredholm, so that the bottom-right entry in the table can be filled.
In Proposition \ref{prop indep}, we show that the index of deformed Dirac operators is independent of choices made. 
%In Theorem \ref{thm induction}, we give an explicit description of the image of this index in $\Rhat(K)$.

\subsection*{Overview}

%In Part \ref{part index}, we introduce $G$-Fredholm operators and the equivariant index, and prove that
%deformed Dirac operators are $G$-Fredholm. In Part \ref{part prop}, we study some properties, examples and applications of the equivariant index of such operators.

We start in Section \ref{sec prelim}, by reviewing some background material on $K$-homology and crossed product algebras. Then we define the index and state the main result in Section \ref{sec ind result}. This result is proved in Sections \ref{sec decomp Dirac} and \ref{sec loc est}. 
%Section \ref{sec induction} contains the statement and proof of the induction result. This, along with other techniques, is then used in Section \ref{sec properties} to study properties, examples and applications of the index.
%We list some notation we will use in
%Appendix \ref{app not}. 
%In Appendix \ref{sec DN K Fred}, we give a simpler proof of the main result in the case of compact groups, but for slightly more general operators. This is used in the induction result. 

\subsection*{Acknowledgements}

The authors thank Maxim Braverman, Nigel Higson, Gennadi Kasparov, %, Bram Mesland, Adam Rennie 
 Mathai Varghese and Guoliang Yu for interesting and helpful discussions. The first author was supported by the European Union, through Marie Curie fellowship PIOF-GA-2011-299300.

%%%%%%%%%%%%%
%%% Preliminaries %%%
%%%%%%%%%%%%%

\section{Preliminaries} \label{sec prelim}

We start by reviewing some basic facts about $K$-homology, some $KK$-theory, and crossed product $C^*$-algebras, for the benefit of readers who are not familiar with these topics. Experts should feel free to skip this section, or have a brief look at our notation and conventions.

We mention an isomorphism in $KK$-theory defined by Morita equivalence. This can be used to identify the equivariant index defined in Section \ref{sec G index} with an object that does not involve $K$-homology and $C^*$-algebras.
%These will be used in some of the applications in \cite{HochsSong16b}. 
For details about $K$-homology, we refer to Higson and Roe's book \cite{higson00}. For the more general $KK$-theory, see Chapter VIII of \cite{Blackadar}.

\subsection{Analytic $K$-homology}
Let $A$ be a separable  $C^*$-algebra. A \emph{Kasparov $(A, \C)$-cycle} is a triple $(\calH, F, \pi)$, where
\begin{itemize}
\item $\calH$ is a $\Z_2$-graded, separable Hilbert space;
\item $F\in \cB(\calH)$ is odd with respect to the grading;
\item $\pi: A\to \cB(\calH)$ is a $*$-homomorphism into the even operators,
\end{itemize}
such that for all $a\in A$, the operators
\beq{F condition}
\pi(a)(F^2-1), \quad [F, \pi(a)] \quad \text{and} \quad \pi(a)(F^*-F)
\eeq
 on $\calH$ are compact.
 
 \begin{definition}
\label{equ-2}
A \emph{unitary equivalence} between two Kasparov $(A, \C)$-cycles
$
(\calH, F, \pi)
$ 
and
$
(\calH', F', \pi')
$ 
is an even unitary isomorphism $\calH \cong \calH'$, which intertwines the representations $\pi$ and $\pi'$ of $A$ and the operators $F$ and $F'$. 
\end{definition}
 
\begin{definition}\label{equiv-1}
Consider two Kasparov $(A, \C)$-cycles $(\calH, F_0, \pi)$ and $(\calH, F_1, \pi)$.
Let $[0,1] \to \cB(\calH)$, denoted by $t\mapsto F_t$, be a norm-continuous path of odd operators. Suppose that for all $t\in [0,1]$, the triple $(\calH, F_t, \pi)$ is a Kasparov $(A, \C)$-cycle. Then $(\calH, F_0, \pi)$ and $(\calH, F_1, \pi)$ are \emph{operator homotopic}.
%Let $\pi: A\to \cB(\calH)$ is a $*$-homomorphism, and suppose that
% $(\calH, F_t,\pi)$ is a Kasparov $(A, \C)$-cycle for all $t$. This is an 
% \emph{operator homotopy} between $(\calH, F_0, \pi)$ and $(\calH, F_1, \pi)$. 
\end{definition}

\begin{definition}\label{def K hom}
The \emph{$K$-homology} of $A$ (in even degree) is the Abelian group $KK(A,\C)$ with one generator for every  class of Kasparov $(A, \C)$-cycles with respect to the equivalence relation generated by unitary equivalence and operator homotopy, subject to the relation
\[
[\calH, F, \pi]+
[\calH', F', \pi'] = \bigl[\calH \oplus \calH', F \oplus F', \pi \oplus \pi'\bigr],
\]
for all equivalence classes $[\calH, F, \pi]$ and $[\calH', F', \pi']$ of Kasparov $(A, \C)$-cycles $(\calH, F, \pi)$ and $(\calH', F', \pi')$, respectively.
%of the set of Kasparov $(A, \C)$-cycles by the equivalence relation generated by operator homotopies and unitary equivalences. This is an Abelian group, with  addition given by direct sum, and the zero element represented by the zero Hilbert space, zero representation, and zero operator.  
\end{definition}
If the operators \eqref{F condition} are zero, then $(\calH, F, \pi)$ is called a \emph{degenerate cycle}, and turns out to represent the zero element in $KK(A, \C)$. 

More generally, if $B$ is another $C^*$-algebra (assumed to be $\sigma$-unital to avoid technical difficulties), then one has the notion of a Kasparov $(A, B)$-cycle. These are defined as Kasparov $(A, \C)$-cycles, with the Hilbert space replaced by a right Hilbert $B$-module.
Similarly to Definition \ref{def K hom}, one obtains the Abelian group $KK(A, B)$.
 If there is a group $G$ acting on $A$ and $B$ in a suitable way, there is an equivariant version as well. We will denote $G$-equivariant $KK$-theory and $K$-homology by a superscript $G$. 

There is also an odd version of $KK$-theory, where there is no $\Z_2$-grading. We will write $KK$ or $KK_0$ for even $KK$-theory and $KK_1$ for odd $KK$-theory, and $KK_*$ for the direct sum of the two. 

The  $KK$-group $KK_*(A, B)$ is covariantly functorial in the first entry, and contravariantly functorial in the second.
If $C$ is a third $C^*$-algebra, there is the Kasparov product
\[
KK_*(A, B) \times KK_*(B, C)\xrightarrow{\otimes_B} KK_*(A, C).
\]
It is associative, and functorial in all natural senses. 

\begin{example}
If $A= \C$, then
\[
KK_0(\C, \C) \cong \Z \quad \text{and} \quad  KK_1(\C, \C) \cong 0. 
\]

\end{example}

\begin{example}\label{ex KK D}
Let $A := C_0(M)$, for a smooth Riemannian manifold $M$. Let $D$ be an elliptic, odd, self-adjoint, first order differential operator on a $\Z_2$-graded, Hermitian vector bundle $\cE \to M$. Let $\pi_M\colon C_0(M) \to \cB(L^2(\cE))$ be given by pointwise multiplication. Then %under a completeness assumption, 
the triple
\beq{eq KK D}
\Bigl(L^2(\cE), \frac{D}{\sqrt{D^2+1}}, \pi_M \Bigr)
\eeq
is a Kasparov $(C_0(M), \C)$-cycle. If a group acts on $M$ and $\cE$, preserving all structure and the operator $D$, then this is an equivariant Kasparov cycle. Its class in the $K$-homology of $C_0(M)$ is denoted by $[D]$. 
% If $D$ is a Dirac-type operator then the completeness condition on $M$ is just metric completeness.
\end{example}

%\begin{example}
%Any  $*$-homomorphism $\psi\colon A\to B$  defines a Kasparov $(A, B)$-cycle $(B, 0, \psi)$. If $A=B$, and $\psi$ is the identity map, the corresponding class in $KK(A, A)$ is denoted by $[1_A]$. The Kasparov product with $[1_A]$ is the identity map.
%\end{example}

\begin{lemma}\label{zero element}
Let $(\calH, F, \pi)$ be a Kasparov $(A, \C)$-cycle. Suppose that there exists a self-adjoint, odd  involution $T$ on $\calH$  that  commutes with the action $\pi$ of $A$, and anticommutes with $F$. Then $(\calH, F, \pi)$ represent the zero element in $KK(A, \C)$. 
\end{lemma}
\begin{proof}
The path $(\calH, F_t, \pi)$ with $F_t = \cos(2\pi t)F + \sin(2\pi t)T$, gives an operator homotopy from $(\calH, F, \pi)$ to a degenerate cycle. 
\end{proof}

\subsection{Crossed product $C^*$-algebras}

Let $A$ be a $C^*$-algebra, and $G$ a locally compact group. Fix a left Haar measure $dg$ on $G$, and  let $\delta_G\colon G \to \R_{\geq 0}$ be the modular function, i.e.\ $d(gg') = \delta_G(g') dg$ for all $g'\in G$.
Suppose there is a homomorphism $G\to \Aut(A)$, continuous with respect to pointwise convergence in norm. We will denote the image of an element $g\in G$ under this map by $g$. The crossed product $A\rtimes G$ is a completion of the $*$-algebra $C_c(G, A)$, with the product and $*$-operation
\[
\begin{split}
(\varphi \varphi')(g) &= \int_G \varphi(g') g' \varphi'(g'^{-1}g)\, dg'; \\
\varphi^*(g) &= \delta_G(g)^{-1} g \varphi(g^{-1})^*,
\end{split}
\]
for $\varphi, \varphi' \in C_c(G, A)$ and $g\in G$. 

The norm in which the completion is taken is defined as follows.
Consider a Hilbert space $\calH$, a unitary representation 
$\pi_G\colon G\to \U(\calH)$ and a $*$-representation
$\pi_A\colon A\to \cB(\calH)$,
such that for all $g\in G$ and $a\in A$,
\[
\pi_G(g) \pi_A(a) \pi_G(g)^* = \pi_A(g a).
\]
These define a $*$-representation
\[
\pi_{G, A}\colon C_c(G, A)\to \cB(\calH),
\]
by 
\beq{eq pi G A}
\pi_{G, A}(\varphi) = \int_G \pi_A(\varphi(g)) \pi_G(g)\, dg,
\eeq
for $\varphi \in C_c(G, A)$. For such $\varphi$, one has
\[
\|\pi_{G, A}(\varphi)\|_{\cB(\calH)} \leq \|\varphi\|_{L^1(G, A)}.
\]
The norm on $A\rtimes G$ is given by
\[
\|\varphi\|_{A\rtimes G} := \sup_{\calH, \pi_G, \pi_A} \|\pi_{G, A}(\varphi)\|_{\cB(\calH)},
\]
where the supremum is taken over $\calH$, $\pi_G$ and $\pi_A$ as above.

If $B$ is another $C^*$-algebra with a continuous action by $G$ as above, and if $\psi\colon A\to B$ is a $G$-equivariant $*$-homomorphism, then we have the induced $*$-homomorphism
\[
\psi_G\colon A\rtimes G \to B\rtimes G,
\]
given by
\beq{eq psi G}
\psi_G(\varphi)(g) = \psi(\varphi(g)),
\eeq
for $\varphi \in C_c(G, A)$ and $g\in G$, and extended continuously. In what follows, we will often work with the dense subalgebra $C_c(G, A)$ (or an even smaller dense subspace), rather than with the complete algebra $A\rtimes G$.

\subsection{Group $C^*$-algebras} \label{sec C*G}

Group $C^*$-algebras are important special cases of crossed products.
\begin{definition}
Let $A = \C$, with the trivial action by $G$. Then $\C \rtimes G$ is the
 \emph{maximal group $C^*$-algebra} $C^*G$ of $G$. It equals  the completion of the convolution algebra $C_c(G)$ with respect to the norm
\[
\|\varphi\|_{\mathrm{max}} := \sup_{\calH, \pi_G} \|\pi_{G}(\varphi)\|_{\cB(\calH)}.
\]
Here the supremum runs over all unitary representations
$\pi_G$ of $G$ in Hilbert spaces $\calH$. For such a representation $\pi_G$, we use the same notation for the $*$-homomorphism $\pi_G\colon C_c(G)\to \cB(\calH)$ given by
\beq{eq piG}
\pi_G(\varphi)v = \int_G \varphi(g) \pi_G(g) v\, dg,
\eeq
for $\varphi \in C_c(G)$ and $v\in \calH$.

The \emph{reduced group $C^*$-algebra $C^*_r G$} is the closure 
\[
\overline{\lambda(C_c(G))} \subset \cB(L^2(G)),
\]
where $\lambda$ denotes the left regular representation of $G$ in $L^2(G)$; i.e.\ for all $\varphi \in C_c(G)$, the operator $\lambda(\varphi)$ is given by convolution by $\varphi$. 
\end{definition}

The $K$-homology group of the group $C^*$-algebra $C^*K$ of a compact group $K$ has a very explicit description. (For compact groups, the maximal and reduced $C^*$-algebras coincide.)
 Let $V$ be any irreducible representation space of $K$. Let $\pi_K\colon C^*K \to \cB(V)$ be given by continuous extension of \eqref{eq piG} (for $G=K$).
%\beq{eq piK}
%\pi_K(\varphi)v = \int_K \varphi(k) k\cdot v\, dk,
%\eeq
%for $\varphi \in C(K) \subset C^*K$ and $v\in V$. (Here $dk$ is the Haar measure giving $K$ unit volume.) 
Consider the grading on $V$ for which all of $V$ is the even part. Then the triple
$
(V, 0, \pi_K)
$
is a Kasparov $(C^*K, \C)$-cycle. This procedure defines an isomorphism of Abelian groups
\beq{eq Khom CK}
KK(C^*K, \C) \cong \Rhat(K).
\eeq
Here 
\[
\Rhat(K)\cong \Hom_{\Z}(R(K), \Z)
\] 
is the completion of the character ring $R(K)$ obtained by allowing infinite linear combinations of irreducible representations, but with finite multiplicities.

The isomorphism \eqref{eq Khom CK} can be described explicitly for more general $K$-homology cycles. 
\begin{lemma} \label{lem Khom CK}
Let $\calH$ be a $\Z_2$-graded, separable Hilbert space, with a unitary representation of $K$. Let $F \in \cB(\calH)$ be an odd, self-adjoint,  $K$-equivariant operator, such that $(\calH, F, \pi_K)$ is a Kasparov $(C^*K, \C)$-cycle. Let $F^{\pm}$ be the restrictions of $F$ to the even and odd parts of $\calH$. Then the representation spaces  $\ker F^{\pm}$ of $K$ define elements of $\Rhat(K)$, and under the isomorphism \eqref{eq Khom CK}, we have
\beq{eq ker F pm}
[\calH, F, \pi_K] = [\ker F^+] - [\ker F^-].
\eeq
\end{lemma}
\begin{proof}
Let $V \in \hat K$, and consider the class $[V] \in KK(\C, C^*K)$. Then, since $F$ is $K$-equivariant,  Example 18.3.2(a) in \cite{Blackadar} implies that
\beq{eq V H}
[V] \otimes_{C^*K}[\calH, F, \pi_K] = \bigl[ (V\otimes \calH)^K, 1_V \otimes F, 1 \bigr]\quad \in KK(\C, \C) = \Z,
\eeq
where $1$ denotes scalar multiplication by complex numbers.
Hence the operator
$
(1_V \otimes F)^2 - 1
$
is compact, so that $1_V \otimes F$ is Fredholm. So its kernel is finite-dimensional, and \eqref{eq V H} equals
\beq{eq F V}
\bigl[\ker (1_V \otimes F), 0, 1\bigr] + \bigl[\ker (1_V \otimes F)^{\perp}, 1_V \otimes F, 1\bigr].
\eeq

Define the operator $\sgn(F)$ by functional calculus. On $\ker (1_V \otimes F)^{\perp}$, the operator $1_V\otimes \sgn(F)$ has the properties of the operator $T$ in Lemma \ref{zero element}. (In particular, its square is the identity.) Hence the second term in \eqref{eq F V} is zero. We conclude that
\[
[V] \otimes_{C^*K}[\calH, F, \pi_K] = \bigl[\ker (1_V \otimes F)^K, 0, 1\bigr] = [\ker F^+: V] - [\ker F^-:V].
\]
Therefore, the multiplicity of $V$ in both sides of \eqref{eq ker F pm} is equal.
%\[
%[\calH, F, \pi_K] = \bigoplus_{V \in \hat K} \bigl([V] \otimes_{C^*K}[\calH, F, \pi_K]\bigr)\cdot [V] = [\ker F^+] - [\ker F^-].
%\]
\end{proof}

\subsection{Morita equivalence} \label{sec Morita}

Let $G$ be a locally compact group, with left Haar measure $dg$ and modular function $\delta_G$. Let $K\subset G$ be a closed subgroup, with Haar measure $dk$. We assume $K$ is unimodular for simplicity; later $K$ will always be compact.

The $C^*$-algebra $C_0(G/K)$ has a natural continuous action by $G$, given by
\[
(g \cdot h)(g'K) = h(g^{-1}g'K),
\]
for $g, g'\in G$ and $h \in C_0(G/K)$. The main examples of crossed products we will use are of the form $C_0(G/K)\rtimes G$. This $C^*$-algebra is Morita equivalent to the group $C^*$-algebra $C^*K$ via a Hilbert $C^*K$-module defined as in Situation 10 in \cite{Rieffel82}.
%$\cM$ defined below, with $*$-representation $\pi_{\cM}\colon C_0(G/K)\rtimes G \to \cB(\cM)$. 
This is Green's imprimitivity theorem, see Proposition 3 on page 203 of \cite{Green78}.
% It implies that the triple
%\[
%(\cM, 0, \pi_{\cM}),
%\]
%where all of $\cM$ belongs to the even part, is a Kasparov $(C_0(G/K)\rtimes G, C^*K )$-cycle, and the corresponding class
%\[
%[\cM] \in KK(C_0(G/K)\rtimes G), C^*K)
%\]
%has an inverse 
%\[
%[\cM]^{-1} \in KK(C^*K, C_0(G/K)\rtimes G).
%\]
%This is to say that
%\[
%\begin{split}
%[\cM] \otimes_{C^*K} [\cM]^{-1} &= [1_{C_0(G/K)\rtimes G}] \quad \in KK(C_0(G/K)\rtimes G, C_0(G/K)\rtimes G);\\
%[\cM]^{-1} \otimes_{C_0(G/K)\rtimes G} [\cM] &= [1_{C^*K}] \quad \in KK(C^*K, C^*K).
%\end{split}
%\]
%Hence for any $C^*$-algebra $A$, the Kasparov products with $[\cM]$ from the left and right define isomorphisms of Abelian groups
%\[
%\begin{split}
%KK(C^*K, A) &\cong KK(C_0(G/K)\rtimes G, A);\\
%KK(A, C_0(G/K)\rtimes G) &\cong KK(A, C^*K).
%\end{split}
%\]
%
%Let us define the module $\cM$, as in Situation 10 in \cite{Rieffel82}. As a Hilbert $C^*K$-module, it is the completion of $C_c(G)$ in the $C^*K$-valued inner product given by
%\[
%(f, f')_{C^*K}(k) = \int_{G}\overline{f({g^{-1}})} f'(g^{-1}k)\, dg,
%\]
%for $f, f' \in C_c(G)$ and $k\in K$. The right action by $C^*K$ is given by
%\[
%(f\psi)(g) = \int_K f(kg) \psi(k) \, dk,
%\]
%for $f \in C_c(G)$, $\psi \in C(K)$ and $g\in G$. The representation $\pi_{\cM}$ is given by
%\[
%(\pi_{\cM}(\varphi)f)(g) = \int_G \varphi(g', gK)f(g'^{-1}g) \delta_G(g')^{1/2}\, dg',
%\]
%for $\varphi \in C_c(G, C_0(G,K))$, $f\in C_c(G)$ and $g\in G$. (We will always identify maps from $G$ to $C_0(G/K)$ with functions on $G\times G/K$.)

The isomorphism
\beq{eq iso ME}
KK(C_0(G/K)\rtimes G, \C) \cong KK(C^*K, \C) \cong \Rhat(K)
\eeq
defined by 
Morita equivalence
%the product with $[\cM]$ 
can be described very explicitly. Let $V\in \hat K$, and consider the representation
\beq{eq pi GKG 1}
\pi_{C_0(G/K)\rtimes G}\colon C_0(G/K)\rtimes G \to \cB\bigl( (L^2(G) \otimes V)^K\bigr), 
\eeq
defined by
\[
(\pi_{C_0(G/K)\rtimes G}(\varphi)\sigma)(g) = \int_G \varphi(g, g'K)\delta_G(g')^{1/2} \sigma(g'^{-1}g)\, dg',
\]
for $\varphi \in C_c(G, C_0(G/K))$, $\sigma \in  (L^2(G) \otimes V)^K$, and $g\in G$. 
%Here $\delta_G$ is the modular function on $G$ as before.
 (On pages 131/132 of \cite{Williams07}, it is explained how different powers of the modular function $\delta_G$ can be used.)
\begin{proposition} \label{prop ME L2}
For $V\in \hat K$, the triple
\[
\bigl( (L^2(G) \otimes V)^K, 0, \pi_{C_0(G/K)\rtimes G}\bigr)
\]
is a Kasparov $(C_0(G/K)\rtimes G, \C)$-cycle. The map 
\[
\Rhat(K) \to KK(C_0(G/K)\rtimes G, \C) 
\]
given by $[V] \mapsto \bigl[(L^2(G) \otimes V)^K, 0, \pi_{C_0(G/K)\rtimes G}\bigr]$ is the isomorphism given by Morita equivalence.
%the product with $[\cM]$.
\end{proposition}
%\begin{proof}
This is  a special case of Proposition 3.11 in \cite{HochsSong16b}. This fact means that the isomorphism \eqref{eq iso ME} is given by an induction procedure from $K$ to $G$. It will not be used in the current paper, it is only included to make the isomorphism \eqref{eq iso ME} more explicit. We therefore postpone its proof to \cite{HochsSong16b}.

\section{The index and the main result} \label{sec ind result}

Let $G$ be a Lie group, with finitely many connected components. Let $K<G$ be a  maximal compact subgroup. 
We will define the notion of a $G$-Fredholm operator, for proper actions by $G$. Such an operator has an equivariant index in the $K$-homology group of the crossed product $C^*$-algebra $C_0(G/K)\rtimes G$, or, via the isomorphisms of Subsection \ref{sec C*G} and \ref{sec Morita}, in $\Rhat(K)$.

One special case of this index is an index of transversally elliptic operators for cocompact actions studied by Kasparov \cite{Kasparov14}. This in turn generalises Atiyah's index of transversally elliptic operators \cite{Atiyah74} in the compact case. Another special case is Braverman's index of deformed Dirac operators \cite{Braverman02}, for compact groups. The main result in this paper is Theorem \ref{thm def Dirac G Fred}, which generalises Braverman's index to noncompact groups. 

For the rest of this paper, we fix a proper, isometric action by $G$
 on a complete Riemannian manifold $M$. Where convenient, we will use the Riemannian metric to identify $T^*M\cong TM$. We denote the space of vector fields on $M$ by $\cX(M)$.
 Let $\cE = \cE^+ \oplus \cE^- \to M$ be a $\Z_2$-graded, Hermitian vector bundle. Suppose the action by $G$ lifts to $\cE$, preserving the grading and the Hermitian metric.

\subsection{The equivariant index} \label{sec G index}

Since $G$ acts properly on $M$, the differentiable version of Abels' theorem, on page 2 of \cite{Abels}, states that there is a smooth, equivariant map
\[
p\colon M\to G/K.
\]
This defines a $*$-homomorphism
\[
p^*\colon C_0(G/K)\to C_b(M),
\]
which induces
\[
p_G^*\colon C_0(G/K) \rtimes G\to C_b(M) \rtimes G,
\]
as in \eqref{eq psi G}. As in \eqref{eq pi G A}, the $*$-representation of $C_b(M)$ on $L^2(\cE)$ by pointwise multiplication, and the unitary representation by $G$ in $L^2(\cE)$, combine to a $*$-representation
\beq{pi G CbM}
\pi_{G, C_b(M)}\colon C_b(M) \rtimes G\to \cB(L^2(\cE)).
\eeq
The representation
\[
\pi_{G, G/K}^p := \pi_{G, C_b(M)} \circ p^*_G\colon C_0(G/K) \rtimes G \to \cB(L^2(\cE))
\]
is given explicitly by
\[
\bigl(\pi_{G, G/K}^p (\varphi)s\bigr)(m) = \int_G \varphi(g, p(m)) g\cdot (s(g^{-1}m))\, dg,
\]
for $\varphi \in C_c(G, C_0(G/K))$, $s \in L^2(\cE)$ and $m\in M$.

Let $F\in \cB(L^2(\cE))$ be an odd, self-adjoint, equivariant operator.
\begin{definition} \label{def G Fredholm}
The operator $F$ is \emph{$G$-Fredholm for $p$} if %for all smooth, equivariant maps $p\colon M\to G/K$ as above, 
the triple
\beq{eq cycle G ind}
(L^2(\cE), F, \pi^p_{G, G/K})
\eeq
is a Kasparov $(C_0(G/K)\rtimes G, \C)$-cycle. Then the \emph{equivariant index}, or \emph{$G$-index}, for $p$ of $F$ is the class
\[
\ind^p_G(F) \in KK(C_0(G/K)\rtimes G, \C)
\]
 of the triple \eqref{eq cycle G ind}.

If $F$ is $G$-Fredholm for all smooth, equivariant maps $p\colon M\to G/K$, then $F$ is \emph{$G$-Fredholm}.
\end{definition}

The crossed product $C_0(G/K)\rtimes G$ contains the dense subspace $C^{\infty}_c(G) \otimes C^{\infty}_c(G/K)$. For $e\in C^{\infty}_c(G)$ and $h\in C^{\infty}_c(G/K)$, we have
\[
\pi^p_{G, G/K}(e\otimes h) = p^*h\, \pi_G(e).
\]
Here $p^*h \in C^{\infty}_b(M)$ is viewed as a pointwise multiplication operator, and $\pi_G(e)$ is defined by
\beq{eq pi G e}
\pi_G(e)s = \int_G e(g) g\cdot s \, dg,
\eeq
for all $s\in L^2(\cE)$.
So, as an equivalent definition,  $F$ is $G$-Fredholm for $p$ if and only if
 for all $e\in C^{\infty}_c(G)$ and $h\in C^{\infty}_c(G/K)$, the operators
$
p^*h \,  \pi_G(e) (F^2-1)
$ 
and
$
[F, p^*h]
$
on $L^2(\cE)$ are compact.

\begin{lemma}
If $F$ is $G$-Fredholm, then
the class $\ind^p_G(F)$ is independent of $p$.
\end{lemma}
\begin{proof}
For $j \in\{0,1\}$, let $p_j\colon M\to G/K$ be  smooth, equivariant maps. Since $G/K$ is $G$-equivariantly contractible, there is a $G$-equivariant homotopy
\[
(p_t\colon M\to G/K)_{t\in [0,1]}
\]
connecting $p_0$ to $p_1$. (The space $G/K$ is a universal example of proper $G$-actions \cite{Connes94}.)
%Then for all $\varphi \in C_0(G/K)\rtimes G$, the operators $\pi^{p_t}_{G, G/K}(\varphi)$ depend continuously on $t$ in the operator norm. It follows from Definition \ref{equiv-1}. 
Set $\tilde \calH := C([0,1], L^2(\cE))$, and define $\tilde F \in \cB(\tilde \calH)$ by applying $F$ after evaluating at a point in $[0,1]$. Define
\[
\tilde \pi_{G, G/K}\colon C_0(G/K) \rtimes G\to \cB(\tilde \calH)
\]
by
\[
(\tilde \pi_{G, G/K} (\varphi) \tilde s) (t) := (\pi^{p_t}_{G, G/K}(\varphi) \tilde s)(t),
\]
for $\varphi \in C_0(G/K) \rtimes G$, $\tilde s \in \tilde \calH$ and $t\in [0,1]$.
Then, the triple $(\tilde \calH, \tilde F, \tilde \pi_{G, G/K})$ defines a homotopy class (a `standard' homotopy in the sense of Definition 17.2.2 in \cite{Blackadar})
\[
[\tilde \calH, \tilde F, \tilde \pi_{G, G/K}] \in KK\bigl(C_0(G/K)\rtimes G, C([0,1]) \bigr),
\]
so
\[
\bigl[L^2(\cE), F, \pi^{p_0}_{G, M} \bigr]= \bigl[L^2(\cE), F, \pi^{p_1}_{G, M} \bigr]
\quad  \in KK(C_0(G/K)\rtimes G, \C).
\]
\end{proof}

Because of this lemma, the following definition makes sense.
\begin{definition}
If $F\in \cB(L^2(\cE))$ is a $G$-Fredholm operator, then its \emph{$G$-index} is the class
\[
\ind_G(F) := \ind_G^p(F) \in KK(C_0(G/K)\rtimes G, \C),
\]
for any smooth, equivariant map $p\colon M\to G/K$.
\end{definition}
 From now on, we will also write $\pi_{G, G/K} := \pi^p_{G, G/K}$ when a map $p$ as above is given, and there is no danger of confusion.

Via the Morita equivalence isomorphism of Subsection \ref{sec Morita}, we can identify the $G$-index of a $G$-Fredholm operator $F$ with an element of $\Rhat(K)$.
%\[
%[\cM]^{-1}\otimes_{C_0(G/K)\rtimes G} \ind_G(F) \in KK(C^*K, \C) = \Rhat(K).
%\]
Furthermore, 
if $G/K$  has an equivariant  $\Spin$ structure, we can use the Dirac induction isomorphism
\[
\DInd_K^G\colon R(K) \xrightarrow{\cong} K_*(C^*_rG)
\]
from the Connes--Kasparov conjecture (see (4.20) in \cite{Connes94})
and the universal coefficient theorem to identify
\[
\Rhat(K) = \Hom_{\Z}(R(K), \Z) =  \Hom_{\Z}(K_*(C^*_rG), \Z) = KK(C^*_rG, \C).
\]
In that way, the $G$-index takes values in the $K$-homology of $C^*_rG$.
% to identify the latter class with
%\[
%\bigl((\DInd_K^G)^*\bigr)^{-1} \bigl( [\cM]^{-1}\otimes_{C_0(G/K)\rtimes G} \ind_G(F) \bigr) \in KK(C^*_rG, \C).
%\]
These identifications will be useful in some of the applications in \cite{HochsSong16b}.
Note that, while the $G$-index can be identified with an element of $\Rhat(K)$, it depends on the action by the whole group $G$. This is apparent from the results in \cite{HochsSong16b}, where, for example,  the $G$-index of certain operators is related to discrete series representations of semisimple groups that have such representations.

One special case of the $G$-index is an index of transversally elliptic operators for cocompact actions studied by Kasparov.
\begin{theorem}[Kasparov]
Suppose $M/G$ is compact, and let $F$ be a properly supported, odd, self-adjoint, equivariant pseudo-differential operator on $\cE$ of order zero, which is transversally elliptic in the sense of Definition 6.1 in \cite{Kasparov14}. Then $F$ is $G$-Fredholm.
\end{theorem}
\begin{proof}
See Proposition 6.4 and Remark 8.19 in \cite{Kasparov14}.
\end{proof}

In the setting of this result, the $G$-index of $F$ is the index defined by Kasparov in Remark 8.19 in \cite{Kasparov14}. If $M$ and $G$ are compact, this reduces to Atiyah's index of transversally elliptic operators \cite{Atiyah74}.

\subsection{Differential operators}

In the setting of Subsection \ref{sec G index}, let $D$ be an elliptic, self-adjoint, odd, equivariant, first order differential operator on $\cE$. 
Let $\sigma_D$ be its principal symbol.
Set
\[
F:=\frac{D}{\sqrt{D^2+1}} \quad \in \cB(L^2(\cE)).
\]
We will use the following criterion for $F$ to be $G$-Fredholm in the proof of our main result, Theorem \ref{thm def Dirac G Fred}.
\begin{lemma} \label{lem diff op G Fred}
Suppose that for a smooth, equivariant map $p\colon M\to G/K$, and all $e\in C^{\infty}_c(G)$ and $h\in C^{\infty}_c(G/K)$, the operator
\beq{eq diff op G Fred}
(D^2+1)^{-1}\pi_G(e) p^*h
\eeq
on $L^2(\cE)$ is compact. Then $F$ is $G$-Fredholm for $p$.
\end{lemma}
\begin{proof}
This follows from Baaj and Julg's description of unbounded $KK$-theory in \cite{Baaj83}. Indeed, the operator \eqref{eq diff op G Fred} is compact if and only if its adjoint
\[
p^*\bar h\,  \pi_G(e^*) (D^2+1)^{-1} = p^*\bar  h\,  \pi_G(e^*) (F^2-1)
\]
is. Hence the second condition in Definition 2.1 in \cite{Baaj83} holds. Furthermore, the commutator
\[
[D, p^*h\,  \pi_G(e)]
\]
equals
\[
\sigma_D(p^*dh) \pi_G(e).
\]
This operator is bounded, so $D$ satisfies all conditions in Definition 2.1 in \cite{Baaj83}. The claim therefore follows from Proposition 2.2 in \cite{Baaj83}.
\end{proof}

The operator $D$ is unbounded, so it is not the kind of operator to which Definition \ref{def G Fredholm} applies. But we will say that $D$ is $G$-Fredholm, or $G$-Fredholm for $p$, if the operator $F$ has the respective property. Then the condition in 
Lemma \ref{lem diff op G Fred} is sufficient for $D$ to be $G$-Fredholm for $p$. If $D$ is $G$-Fredholm for $p$, then we write
\[
\ind_G^p(D) := \ind_G^p(F) \quad \in KK(C_0(G/K)\rtimes G, \C).
\]
In particular, in the setting of Lemma \ref{lem diff op G Fred}, 
we have the \emph{spectral triple} $\bigl(C^{\infty}_c(G \times G/K), L^2(\cE), D\bigr)$. If $D$ is $G$-Fredholm (i.e.\ for all such maps $p$), we write
\[
\ind_G(D) := \ind_G(F) \quad \in KK(C_0(G/K)\rtimes G, \C).
\]
If $D$ is only essentially self-adjoint, then we define the $G$-Fredholm property for $D$, and its equivariant index, in terms of its self-adjoint closure.

\begin{example}
Suppose $M/G$ is compact. Then the function $p^*h$ is compactly supported for all $h\in C^{\infty}_c(G/K)$ and all smooth, equivariant maps $p\colon M\to G/K$.
If $D$ is elliptic, then the Rellich lemma therefore implies that for
all $e\in C^{\infty}_c(G)$ and $h\in C^{\infty}_c(G/K)$, the operator
\[
(D^2+1)^{-1}\pi_G(e) p^*h = \pi_G(e)  (D^2+1)^{-1}p^*h
\]
is compact. Hence $D$ is $G$-Fredholm. This is a (trivial) special case of Proposition 6.4 and Remark 8.19 in \cite{Kasparov14}, for transversally elliptic operators. In the elliptic case, this is analogous to the usual argument that the triple \eqref{eq KK D} is a Kasparov cycle.
\end{example}
This example shows that all challenges in investigating which elliptic differential operators are $G$-Fredholm come from cases where $M/G$ is noncompact.

\subsection{Deformed Dirac operators} \label{sec def Dirac}

Braverman \cite{Braverman02} developed equivariant index theory of
%One area of equivariant index theory for non-cocompact actions was developed by Braverman \cite{Braverman02}. He studied 
deformed Dirac operators for actions by compact groups on possibly noncompact manifolds. We will see that this theory fits into the framework of $G$-Fredholm operators, where it generalises to noncompact groups.

Let us define the deformed Dirac operators considered by Braverman. (They already played an important role on compact manifolds in \cite{Zhang98}.) Let $M$, $G$ and $\cE$ be as in Subsection \ref{sec G index}. From now on, we suppose there is a vector bundle homomorphism
 \[
 c\colon TM \to \End(\cE),
 \]
whose image lies in the skew-adjoint, odd endomorphisms, such that for all $v \in TM$,
\[
c(v)^2 = -\|v\|^2.
\]
Then $\cE$ is called a \emph{Clifford module} over $M$, and $c$ is called the \emph{Clifford action}.

A \emph{Clifford connection} is a Hermitian connection $\nabla^{\cE}$ on $\cE$ that preserves the grading on $\cE$, such that for all vector fields $v, w \in \cX(M)$,
\[
[\nabla^{\cE}_v, c(w)] = c(\nabla^{TM}_v w),
\]
where $\nabla^{TM}$ is the Levi--Civita connection on $TM$. We will identify $TM \cong T^*M$ via the Riemannian metric. Then the Clifford action $c$ defines a map
\[
c\colon \Omega^1(M; \cE) \to \Gamma^{\infty}(\cE). 
\]
The \emph{Dirac operator} $D$ associated to a Clifford connection $\nabla^{\cE}$ is defined as the composition
\begin{equation} \label{eq def Dirac}
D\colon \Gamma^{\infty}(\cE) \xrightarrow{\nabla^{\cE}} \Omega^1(M; \cE) \xrightarrow{c} \Gamma^{\infty}(\cE).
\end{equation}
In terms of a local orthonormal frame $\{e_{1}, \ldots, e_{\dim M}\}$ of $TM$, one has
\begin{equation} \label{eq Dirac local}
D = \sum_{j=1}^{\dim M} c(e_j) \nabla^{\cE}_{e_j}.
\end{equation}
This operator interchanges sections of $\cE^+$ and $\cE^-$. We will denote the restriction of $D$ to $\Gamma^{\infty}(\cE^{\pm})$ by  $D^{\pm}$. 

Suppose that 
for all $g \in G$, $m \in M$, $v \in T_mM$ and $u \in \cE_m$ we have\footnote{In fact, this condition implies that the action by $G$ preserves the Riemannian metric.}
\[
g\cdot  c(v)u = c(g\cdot v) g\cdot u.
\]
Then $\cE$ is called a \emph{$G$-equivariant Clifford module} over $M$. In this case, the Dirac operator associated to a $G$-invariant Clifford connection is $G$-equivariant.  
We fix  a $G$-invariant Clifford connection $\nabla^{\cE}$ on $\cE$ for the rest of this paper, and consider the Dirac operator $D$ associated to $\nabla^{\cE}$.

If $M$ and $G = K$ are \emph{compact}, then the kernel of the Dirac operator $D$ is finite-dimensional, and we have its equivariant index
\[
\ind_K(D)  = [\ker D^+] - [\ker D^-] \in R(K).
\]
More generally, if $M/G$ is compact, we can apply the analytic assembly map   $\mu_M^G$ from \cite{Connes94} to the $K$-homology class $[D]$ as in Example \ref{ex KK D}, to obtain an index
\[
\mu_M^G[D] \in KK(\C, C^*G).
\]
Our goal in this paper is to develop index theory for cases where both $G$ and $M/G$ are noncompact, however.

To define an index when $M/G$ is noncompact, we consider a smooth, equivariant map
\[
\psi\colon M\to \kg.
\]
It induces a vector field $v^{\psi} \in \cX(M)$, defined by
\[
v^{\psi}_m = \ddt \exp(-t\psi(m))\cdot m,
\]
for all $m\in M$. This vector field is $G$-invariant.
\begin{definition}\label{def def Dirac}
The \emph{Dirac operator deformed by $\psi$} is the operator
\[
D_{\psi} = D - \ii c(v^{\psi})
\]
on $\Gamma^{\infty}(\cE)$. %We will write $D^{\pm}_{\psi}$ for the even and odd parts of this operator.
\end{definition}

Let $\Zeroes(v^{\psi})\subset M$ be the set of zeroes of $v^{\psi}$.
\begin{assumption}
The set $\Zeroes(v^{\psi})$ is cocompact; i.e. $\Zeroes(v^{\psi})/G$ is compact. 
\end{assumption}

If $G$ is compact, Braverman defined  an equivariant index of the Dirac operator deformed by $f\psi$, for a function $f$ that is \emph{admissible} in the following sense.
\begin{definition} \label{def admissible}
Let a real-valued function $\rho \in C^{\infty}(M)^G$ be given. A nonnegative function $f\in C^{\infty}(M)^G$ is \emph{$\rho$-admissible} if, outside a cocompact subset of $M$, we have
\[
\frac{f^2}{\|df\| + f + 1} \geq \rho.
\]
\end{definition}
This property of a function $f$ reflects that it grows fast enough compared to its derivative. (Braverman's notion of admissibility, as in Definition 2.6 in \cite{Braverman02}, is slightly different. The one we use is sufficient, however.) Admissible functions always exist.
\begin{lemma}\label{lem global f}
For any real-valued $\rho \in C^{\infty}(M)^G$, there exist $\rho$-admissible functions.
\end{lemma}
\begin{proof}
Without loss of generality, we may assume that $\rho \geq 1/4$.
By Lemma C.3 in \cite{Mathai13}, there is
a positive, $G$-invariant smooth function $f$ such that
\[
\begin{split}
f^{-1} &\leq \rho^{-1}/4 \\
\|d(f^{-1})\|&\leq \rho^{-1}/2.
\end{split}
\]
 Then $f^{-1} \leq 1$, so $f^{-2} \leq f^{-1}$, and hence
\[
\frac{ f^2}{\|df\| + f  + 1} = \bigl(\|d (f^{-1})\| + f^{-1} + f^{-2} \bigr)^{-1} \geq \bigl(\|d (f^{-1})\| + 2f^{-1} \bigr)^{-1} \geq \rho.
\]
\end{proof}

\begin{theorem}[Braverman] \label{thm Braverman}
Suppose $G=K$ is compact. Then
there is a real-valued function $\rho_{\Br} \in C^{\infty}(M)^K$ such that for all $\rho_{\Br}$-admissible functions $f\in C^{\infty}(M)^K$, and all
 irreducible representations $V$ of $K$, the multiplicity $m^{\pm}_V$ of $V$ in
\[
\ker D_{f\psi} \cap L^2(\cE^{\pm})
\]
is finite. The index
\[
\ind_K^{\Br}(D_{f\psi}) := \sum_{V\in \hat K} (m^+_V - m^-_V)V \quad \in \Rhat(K)
\]
is independent of the $\rho_{\Br}$-admissible function $f$, the connection $\nabla^{\cE}$ and the complete Riemannian metric on $M$.
\end{theorem}
 This is Theorem 2.9 in \cite{Braverman02}. Independence of the various choices follows from a general cobordism invariance result for this index, Theorem 3.7 in \cite{Braverman02}. Note that if $G=K$, there is only one map $M \to G/K$.
 
Equivariant index theory of deformed Dirac operators is relevant for example to geometric quantisation. Already in the compact case, deformed Dirac operators were used by Tian and Zhang in \cite{Zhang98} to obtain a localisation result for the index of a Dirac operator.  For compact groups and noncompact manifolds, they were used in  \cite{HochsSong15,  Zhang14, Paradan11}. 
The indices used in \cite{Zhang14, Paradan11} are defined differently, but are equal to Braverman's. 
The fact that these three indices turn out to be equal is an indication that they are natural objects to study.
Another reason why it is natural to consider deformed Dirac operators is that the deformation term arises from basic constructions in certain cases. For example,
 for $\Spinc$-Dirac operators, the deformation just amounts to a different choice of connection (see Remark 3.7 in \cite{HochsSong15}). 
 %And if $v^{\psi}$ is the gradient of  a smooth function $\varphi$ on $M$, then $D_{\psi} = e^{\ii \varphi} D e^{-\ii \varphi}$.
 More generally, it is interesting to investigate a class of equivariant elliptic operators for non-cocompact actions that have well-defined equivariant indices. A complication for the deformed Dirac operators studied here is that the anticommutator $Dc(v^{\psi}) + c(v^{\psi})D$ is not a vector bundle endomorphism of $\cE$, but has a first order part (see Lemma \ref{lem Bochner}). This is in contrast to Callias-type deformations of Dirac operators, see \cite{Anghel93, Bruening92b, Bunke95, Callias78, Kucerovsky01}.
 
\subsection{Noncompact groups; the main result} 
 
We now allow $G$ to be noncompact. 
In \cite{Braverman14, Mathai13}, one studies an index of deformed Dirac operators that only involves $G$-invariant sections of $\cE$, i.e.\ the isotypical component of the trivial representation. It is  an interesting and natural question if this can be extended to nontrivial representations.
However, it is not clear a priori how to do this, or even where such an index should take values. For one thing,  the unitary dual of a noncompact group is not discrete. For another, the nontrivial irreducible representations of a noncompact simple group are infinite-dimensional, which means Braverman's arguments in \cite{Braverman02} do not apply to nontrivial representations. 
 
 The techniques used in \cite{Braverman14, Mathai13} fundamentally only apply to $G$-invariant sections, but
the notion of $G$-Fredholm operators makes a completely different approach possible. This
 allows us to generalise Braverman's index to noncompact groups and nontrivial representations. This is the main result in this paper.
\begin{theorem}[Deformed Dirac operators are $G$-Fredholm] \label{thm def Dirac G Fred}
Let $p\colon M\to G/K$ be smooth and equivariant. 
There is a real-valued function $\rho \in C^{\infty}(M)^G$ (depending on the Riemannian metric on $TM$, the connection on $\cE$ used to define $D$, and the map $p$) such that the operator $D_{f\psi}$ is $G$-Fredholm for $p$, for all $\rho$-admissible functions $f$. 
\end{theorem}
In the setting of this result, we have the $G$-index
\[
\ind^p_G(D_{f\psi}) \in KK(C_0(G/K) \rtimes G, \C),
\]
for $\rho$-admissible functions $f$. This index is independent of the map $p$, the Riemannian metric on $M$, the function $f$, and the connection $\nabla^{\cE}$. Because the function $\rho$ depends on these data, we need to be somewhat careful in the precise formulation of this independence property.
\begin{proposition} \label{prop indep}
For $j=0,1$, let a
 complete, $G$-invariant Riemannian metric $B_j$ on $TM$ be given, and a corresponding $G$-equivariant Clifford action $c_j$ by $TM$ on the vector bundle $\cE$. Write $\cE_j$ for the vector bundle $\cE$ with this Clifford action.
 Fix a
 $G$-invariant Clifford connection $\nabla^{\cE_j}$ on $\cE_j$, and a smooth, equivariant map
   $p_j\colon M\to G/K$. Let $D^{\cE_j}$ be the Dirac operator associated to the connection $\nabla^{\cE_j}$. Let $\rho_j$ be a function as in Theorem \ref{thm def Dirac G Fred}, for the data $(B_j, \cE_j, \nabla^{\cE_j}, p_j)$.  Let 
 $f_j$ be a $\rho_j$-admissible function. Consider the deformed Dirac operator
 \[
 D^{\cE_j}_{f_j\psi} = D^{\cE_j} - \ii f_j c_j(v^{\psi}).
 \]
 Then
\[
\ind_G^{p_0}(D^{\cE_0}_{f_0\psi}) = \ind_G^{p_1}(D^{\cE_1}_{f_1\psi}) \quad \in KK(C_0(G/K)\rtimes G, \C). 
\]
\end{proposition}
We write
\[
\ind_G(\cE, \psi) := \ind_G^{p}(D_{f\psi}),
\]
for any complete, $G$-invariant Riemannian metric on $TM$, $G$-invariant Clifford connection, smooth equivariant map $p\colon M\to G/K$ and $\rho$-admissible function $f$.
Because of Lemma \ref{lem Khom CK}, this index reduces to Braverman's index if $G$ is compact.

As far as the authors know, there is currently no
 other version of equivariant index theory for noncompact groups and orbit spaces. Here we interpret an equivariant index as taking values in an object defined purely in terms of the group acting, such as $KK(C_0(G/K)\rtimes G, \C)$. 
In the induction result in \cite{HochsSong16b}, we give an explicit description of the image of the $G$-index of a deformed Dirac operator in $\Rhat(K)$.
In \cite{HochsSong16b}, we also give some relations between the $G$-index and existing indices in cases where  $M/G$ is compact, and some properties and applications in the general case.

\subsection{Idea of the proof}\label{sec idea}

We prove Theorem \ref{thm def Dirac G Fred} and Proposition \ref{prop indep} in Sections \ref{sec decomp Dirac} and \ref{sec loc est}. Here we describe the idea of the proof of Theorem \ref{thm def Dirac G Fred}.

First of all, it is important that the operator $D_{f\psi}$ is essentially self-adjoint. This follows 
from
Proposition 10.2.11 in \cite{higson00}, which we restate here.
\begin{proposition} \label{prop ess sa}
If $A$ is a symmetric, first order differential operator on $\cE\to M$, with principal symbol $\sigma_A$, and if its propagation speed
\[
\sup\{ \|\sigma_A(v)\|; \text{$v\in TM$ and $\|v\|=1$}\}
\]
is finite, then $A$ is essentially self-adjoint.
\end{proposition} 

To apply Lemma \ref{lem diff op G Fred}, we
choose an open cover $\{U_j\}_{j=0}^{\infty}$ of $M$ by $G$-invariant, relatively cocompact open sets $U_j$. We choose these sets so that for all $j$, the boundary $\partial U_j$ is a smooth submanifold of $M$, and has a neighbourhood in $U_j$ diffeomorphic to $\partial U_j \times [0, 1[$.
Because $\Zeroes(v^{\psi})/G$ is compact, we can choose this cover so that $\|v^{\psi}\|$ has a positive lower bound on $U_j$ for all $j\geq 1$. Then in particular, $\Zeroes(v^{\psi}) \subset U_0$. In addition, we choose this cover so that every point in $M$ is contained in only finitely many of the sets $U_j$.

Let $\{\chi_j\}_{j=0}^{\infty}$ be a sequence of $G$-invariant functions such that $\supp (\chi_j)\subset U_j$ for all $j$, and $\{\chi_j^2\}_{j=0}^{\infty}$ is a partition of unity. Let $p\colon M\to G/K$ be a smooth, equivariant map. Then for all $f\in C^{\infty}(M)^G$, and all $e\in C^{\infty}_c(G)$ and $h\in C^{\infty}_c(G/K)$, we have
\beq{eq sum G Fred}
(D_{f\psi}^2+1)^{-1}\pi_G(e) p^*h = \sum_{j=0}^{\infty} (D_{f\psi}^2+1)^{-1}\pi_G(e) p^*h \, \chi_j^2.
\eeq
Because the function $p^*h\, \chi_j^2$ has compact support for all $j$, the Rellich lemma implies that all terms in the sum on the right hand side are compact operators. Therefore, the operator $D_{f\psi}$ is $G$-Fredholm if $f$ is $\rho$-admissible, for $\rho$ as in the following proposition.
\begin{proposition} \label{prop sum G Fred}
There is a real-valued function $\rho \in C^{\infty}(M)^G$, such that for all
 $e\in C^{\infty}_c(G)$ and all $h\in C^{\infty}_c(G/K)$, there is a constant $B_{e, h}$, such that 
 for 
 %all smooth, equivariant maps $p\colon M\to G/K$, 
 all $\rho$-admissible functions $f$, and all $j\geq 1$, 
\beq{eq norm sum Fred}
\|(D_{f\psi}^2+1)^{-1}\pi_G(e) p^*h\, \chi_j^2 \| \leq 2^{-j} B_{e, h}.
\eeq
\end{proposition}
If $f$ is $\rho$-admissible for such a function $\rho$, then the sum \eqref{eq sum G Fred} of compact operators converges in the operator norm, to a compact operator. Hence $D_{f\psi}$ is $G$-Fredholm, by Lemma \ref{lem diff op G Fred}.

The idea behind the proof of Proposition \ref{prop sum G Fred} is to show that we have
\[
D_{f\psi}^2 \pi_G(e)p^*h \, \chi_j^2 = (D^2 + f^2 \|v^{\psi}\|^2) \pi_G(e)p^*h\, \chi_j^2 + A_j
\]
where $A_j$ is a bounded operator. For $j\geq 1$, the term $f^2 \|v^{\psi}\|^2$ is large on $U_j$ if $f$ is large there, so that the norm on the left hand side of \eqref{eq norm sum Fred} is small in an appropriate sense. Making this idea precise turns out to require a more elaborate argument than the authors had expected initially.

It is important that the function $\rho$ does not depend on $e$ and $h$ in Proposition \ref{prop sum G Fred}. For this reason, we will need to carefully distinguish between constants depending on $e$ and $h$, and constants depending on other data, in the estimates in Sections \ref{sec decomp Dirac} and \ref{sec loc est}.

%%%%%%%%%%%%%%%%%%%
%%% Decomposing Dirac ops %%%
%%%%%%%%%%%%%%%%%%%

\section{Decomposing Dirac operators} \label{sec decomp Dirac}

This section contains some preparatory material for the proof of Proposition \ref{prop sum G Fred} in Section \ref{sec loc est}. We will decompose the square of a deformed Dirac operator, and use the fact that some terms in this decomposition only differentiate in orbit directions. These terms are \emph{$G$-differential operators} in the sense of Subsection \ref{sec G diff}, which means we can apply a general estimate for such operators. We will also use embeddings of open subsets of $M$ into complete manifolds, as discussed in Subsection \ref{sec cylinder}. This will make certain locally defined operators invertible.

Fix a smooth, equivariant map $p\colon M\to G/K$. It is automatically a submersion, so that
\[
N := p^{-1}(eK) \subset M
\]
is a smooth, $K$-invariant submanifold. Furthermore, the map $G\times N\to M$, mapping an element $(g, n) \in G\times N$ to $gn$, descends to a $G$-equivariant diffeomorphism
\beq{eq M G K}
G\times_K N \xrightarrow{\cong} M.
\eeq
%Here $G\times_K N$ is the quotient of $G\times N$ by the action by $K$ defined by
%\[
%k\cdot (g, n) = (gk^{-1}, kn),
%\]
%for  $k\in K$, $g\in G$ and $n\in N$.

We also fix %a $G$-invariant Clifford connection $\nabla^{\cE}$ on $\cE$, 
an equivariant map $\psi\colon M \to \kg$ for which $\Zeroes(v^{\psi})/G$ is compact, and a real-valued function $f \in C^{\infty}(M)^G$. Let 
% $D$ be the Dirac operator associated to $\nabla^{\cE}$, and 
$D_{f\psi}$ be the deformed Dirac operator as in Definition \ref{def def Dirac}.

\subsection{$G$-differential operators} \label{sec G diff}

Before analysing deformed Dirac operators, we obtain an estimate for operators whose highest-order parts only differentiate in orbit directions. In this subsection and the next, $\cE \to M$ is a $G$-equivariant, Hermitian vector bundle as before, but we will not use the Clifford action for now.

For an element $X \in \kg$, we denote the induced vector field on $M$ by $X^M$. Our sign convention is that for all $m\in M$,
\[
X^M_m = \ddt \exp(-tX)\cdot m.
\]
The Lie derivative of sections of $\cE$ with respect to $X$ is denoted by $\calL_X$. Let $\tau_M\colon T^*M \to M$ be the cotangent bundle projection.

Let $T^*_GM\subset T^*M$ be the subset of elements that annihilate tangent vectors to orbits:
\[
T^*_GM = \bigl\{\xi\in T^*M; \langle\xi, X^M_{\tau_M(\xi)}\rangle = 0 \text{ for all $X\in \kg$}\bigr\}.
\]
If $G$ is compact, a differential operator is transversally elliptic if its principal symbol is invertible outside a compact subset of $T^*_G M$. In a sense, $G$-differential operators have the opposite property.
\begin{definition}
A differential operator $A$ on $\Gamma^{\infty}(\cE)$, with principal symbol $\sigma_A$, is a \emph{$G$-differential operator} if $\sigma_A$ is zero on $T^*_GM$.
\end{definition} 
It follows immediately that, if $G$ is compact, then for any transversally elliptic operator $B$ and any $G$-differential operator $A$, the operator $A+B$ has the same symbol class in $K$-theory as $B$:
\[
[\sigma_{A+B}] = [\sigma_B] \in K^0(T^*_GM).
\]

A $G$-differential operator of order at least $2$ can have lower-order terms that do not just differentiate in orbit directions. We will only consider first order $G$-differential operators, however, and these can be described entirely in terms of differentiation along orbits. Fix a basis $\{X_1, \ldots, X_{\dim G}\}$ of $\kg$.
\begin{lemma}\label{lem GDO 1}
A first order differential operator on $\cE$ is a $G$-differential operator if and only if there are vector bundle endomorphisms $a_j$ and $b$ of $\cE$ such that
\beq{eq GDO 1}
A = \sum_{j=1}^{\dim G} a_j \calL_{X_j} + b.
\eeq
\end{lemma}
\begin{proof}
Let $A$ be a first order differential operator on $\cE$. If $A$ is of the form \eqref{eq GDO 1}, then for all $m\in M$ and $\xi \in (T^*_GM)_m$,
\[
\sigma_{A}(\xi) =  \sum_{j=1}^{\dim G} (a_j)_m  \langle\xi, (X_j^M)_m\rangle = 0.
\]

Conversely, suppose $A$ is a $G$-differential operator. For all $j$, let $(X_j^M)^* \in \Omega^1(M)$ be dual to $X_j^M$ with respect to the Riemannian metric. Define the differential operator
\[
\tilde A := \sum_{j=1}^{\dim G} \sigma_A\bigl((X_j^M)^*\bigr)\calL_{X_j}
\]
on $\cE$. Then for all $m\in M$ and $\xi \in T^*M$,
\[
\sigma_{\tilde A}(\xi) =  \sigma_A \bigl((X_j^M)^*_m\bigr)\langle \xi, (X_j^M)_m\rangle = \sigma_A(\xi).
\]
Hence $b:= A - \tilde A$ is a vector bundle endomorphism. 
\end{proof}

\subsection{An estimate for $G$-differential operators} \label{sec est G diff}

One ingredient of the proof of Proposition \ref{prop sum G Fred} is an estimate for $G$-equivariant, first order $G$-differential operators. 
In the proof of this estimate, we will use certain compact subsets of $G$. 
\begin{lemma}\label{lem Seh}
For all $e\in C^{\infty}_c(G)$ and $h \in C^{\infty}_c(G/K)$, there is a compact subset $S_{e, h} \subset G$, independent of the map $p$, such that for all $\tilde e \in C^{\infty}_c(G)$ with support inside $\supp e$, and all $s\in \Gamma(\cE)$, $g\in G\setminus S_{e, h}$ and $n\in N$, we have
\[
(\pi_G(\tilde e)p^*h \, s)(gn) = 0.
\]
\end{lemma}
\begin{proof}
Let $e\in C^{\infty}_c(G)$ and $h \in C^{\infty}_c(G/K)$ be given. Let $q\colon G\to G/K$ be the quotient map. Set
\[
S_{e, h} := \{g\in G; \supp(e)\cap g \bigl(q^{-1}(\supp(h))\bigr)^{-1} \not=\emptyset \}.
\]
Then, if  $\tilde e\in C^{\infty}_c(G)$ is supported in $\supp e$, we have for all  $s\in \Gamma(E)$, $g\in G$ and $n\in N$,
\[
(\pi_G(\tilde e)p^*h\, s)(gn) = \int_G \tilde e(g') h(g'^{-1}gK) g'(s(g'^{-1}gn))\, dg'.
\]
If $g \not\in S_{e,h}$, then $\tilde e(g') h(g'^{-1}gK) = 0$ for all $g'\in G$.
\end{proof}

For all $X\in \kg$, let $L_X$ be the operator on $C^{\infty}(G)$ defined by the infinitesimal left regular representation. Then for all $e\in C^{\infty}_c(G)$,
\beq{eq LXe}
\calL_X \circ \pi_{G}(e) = \pi_G(L_X(e)).
\eeq
This will be used to prove the estimate for $G$-differential operators we need.

\begin{proposition}\label{prop est GDO}
For any $e\in C^{\infty}_c(G)$ and any $h\in C^{\infty}_c(G/K)$, there is a constant $B_{e, h} > 0$, such that for any $G$-equivariant, first order $G$-differential operator $A$ on $\cE$, with $a_j$ and $b$ as in Lemma \ref{lem GDO 1} bounded on $N$, there is a constant $C_{A, p} > 0$, independent of $e$ and $h$, such that the operator %for all $s\in L^2(\cE)$,
\[
A \pi_G(e)p^*h 
\] 
on $L^2(\cE)$
is bounded, with norm at most $B_{e, h} C_{A,p}$.
\end{proposition}
\begin{proof}
Let  $e\in C^{\infty}_c(G)$ and  $h\in C^{\infty}_c(G/K)$ be given, and let $S_{e, h}\subset G$ be as in Lemma \ref{lem Seh}. For $j, k = 1, \ldots, \dim G$, let $\Ad^k_j \in C^{\infty}(G)$ be the functions such that for all $j$, and all $g\in G$,
\[
\Ad(g)X_j = \sum_{k=1}^{\dim G} \Ad^k_j(g)X_k.
\]
Set
\[
\|\Ad\|_{S_{e, h}} := \max_{g\in S_{e, h}} \max_{j, k =1, \ldots, \dim G }|\Ad^k_j(g)|,
\]
and
\[
B_{e, h} :=\max\Bigl\{ \|e\|_{L^1(G)} \|h\|_{\infty}, \dim G \cdot \|\Ad\|_{S_{e, h}} \cdot \sum_{j=1}^{\dim G} \|L_{X_j}(e)\|_{L^1(G)} \|h\|_{\infty}\Bigr\}.
\]

Let $A$ be a $G$-equivariant $G$-differential operator. Write
\[
A = \sum_{j=1}^{\dim G} a_j \calL_{X_j} + b
\]
as in Lemma \ref{lem GDO 1}. By assumption, the pointwise norms of the vector bundle endomorphisms $a_j$ and $b$ are bounded on $N$. Set
\[
\begin{split}
C_{a, p} &:=  \sum_{j=1}^{\dim G} \sup_{n\in N} \|(a_j)_n\|;\\
C_{b, p} &:= \sup_{n\in N} \|b_n\|; \\
C_{A, p} &:= C_{a, p} + C_{b, p}.
\end{split}
\]

%Since $A$ is $G$-equivariant, so is $b$. Therefore, we have for all $m\in M$,
%\[
%\|b_m\| \leq C_{b, p}.
%\]
%Furthermore,
%\[
% \|\pi_G(e) \|_{\cB(L^2(\cE))} \leq \|e\|_{L^1(G)},
%\]
%and
%\[
%\|p^*h\|_{\cB(L^2(\cE))} \leq \|h\|_{\infty}.
%\]
%We find that
%\[
%\|b \pi_G(e) p^*h\|_{\cB(L^2(\cE))} \leq  B_{e,  h} C_{b, p} .
%\]
%
%Next, we consider the first order part
%\[
%A_1 := \sum_{j=1}^{\dim G} a_j
%\]
%of $A$. 
The operators $a_j$ are not $G$-equivariant in general, even though the whole operator $A$ is. Let $s\in \Gamma^{\infty}_c(\cE)$, $g\in G$ and $n\in N$ be given. Then equivariance of $A$ implies that
\[
\begin{split}
(g^{-1}As)(n) &= (Ag^{-1}s)(n)\\
%% Can leave out %%
%	&= \sum_{j=1}^{\dim G} (a_j)_n \bigl(\calL_{X_j} g^{-1}s \bigr)(n) + b_n (g^{-1}s)(n)\\
%	&= \sum_{j=1}^{\dim G} (a_j)_n \bigl(g^{-1} \calL_{\Ad(g)X_j} s \bigr)(n) + b_n (g^{-1}s)(n) \\
%% up to here %%	
	&= \sum_{j=1}^{\dim G} (a_j)_ng^{-1} \bigl( \calL_{\Ad(g)X_j} s (gn) \bigr) +  b_n g^{-1}(s(gn)).\\
\end{split}
\] 
So
\[
(As)(gn) = \sum_{j=1}^{\dim G} (g (a_j)_n g^{-1}) \bigl( \calL_{\Ad(g)X_j} s (gn) \bigr) + (g b_n g^{-1}) s(gn).
\]
Replacing $s$ by $\pi_G(e)p^*h\, s$,
we obtain the pointwise estimate
\[
\begin{split}
\bigl\| (A \pi_G(e)p^*h\, s)(gn) \bigr\| &\leq  \sum_{j=1}^{\dim G} \| (a_j)_n \|  \bigl\| \calL_{\Ad(g)X_j} \pi_G(e)p^*h\, s (gn) \bigr\| + \|b_n\| \|\pi_G(e)p^*h\, s(gn)\| \\
	&\leq C_{a, p} \sum_{j=1}^{\dim G}\bigl\| \calL_{\Ad(g)X_j} \pi_G(e)p^*h\, s (gn) \bigr\| + C_{b, p} \|e\|_{L^1(G)} \|h\|_{\infty} \|s(gn)\|.
\end{split}
\]
For every $j$, we have
\[
\begin{split}
\bigl\| \calL_{\Ad(g)X_j} \pi_G(e)p^*h\, s (gn) \bigr\| &\leq \sum_{k=1}^{\dim G} |\Ad^k_j(g)| \bigl\| \calL_{X_k} \pi_G(e)p^*h\, s (gn) \bigr\| \\
&\leq \sum_{k=1}^{\dim G} |\Ad^k_j(g)| \bigl\| \pi_G(L_{X_k} (e) )p^*h\, s (gn) \bigr\|, 
\end{split}
\]
where we used \eqref{eq LXe}.

The pointwise estimates obtained so far imply that
\begin{multline*}
\|A \pi_G(e)p^*h\, s \|_{L^2(\cE)}  \\\leq
C_{a, p} \dim (G) \|\Ad \|_{S_{e, h}} \sum_{k=1}^{\dim G}\bigl\| \pi_G(L_{X_k} (e) )p^*h\, s \bigr\|_{L^2(\cE)} 
 + C_{b, p} \|e\|_{L^1(G)} \|h\|_{\infty} \|s\|_{L^2(\cE)} 
\\ \leq 
B_{e, h} C_{A,p} \|s\|_{L^2(\cE)}.
\end{multline*}

\end{proof}

\subsection{The square of a deformed Dirac operator}

The first two steps in the proof of Proposition \ref{prop sum G Fred} are a decomposition of the square of the deformed Dirac operator $D_{f\psi}$, and a decomposition of the undeformed Dirac operator $D$. In both of these decompositions, $G$-differential operators appear. In Section \ref{sec loc est},   we will apply Proposition \ref{prop est GDO} to those operators.

\begin{lemma}\label{lem Bochner}
For any local orthonormal frame $\{e_1, \ldots, e_{\dim M}\}$ of $TM$, we have, locally,
\[
D_{f\psi}^2 = D^2 + f^2\|v^{\psi}\|^2 - \ii \sum_{j=1}^{\dim M} c(e_j)c(\nabla^{TM}_{e_j} fv^{\psi}) + 2\ii f\nabla^{\cE}_{v^{\psi}}.
\]
\end{lemma}
\begin{proof}
We have
\[
D_{f\psi}^2 = D^2 + f^2\|v^{\psi}\|^2 - \ii(D c(fv^{\psi}) + c(fv^{\psi}) D),
\]
and, locally,
\[
\begin{split}
D c(fv^{\psi}) + c(fv^{\psi}) D 
%% Can leave out %%
%&= \sum_{j=1}^{\dim M}\Bigl( c(e_j) \nabla^{\cE}_{e_j}  c(fv^{\psi}) + c(fv^{\psi}) c(e_j) \nabla^{\cE}_{e_j}\Bigr) \\
%	&= \sum_{j=1}^{\dim M}\Bigl( c(e_j)  \bigl(c(fv^{\psi}) \nabla^{\cE}_{e_j} + c(\nabla^{TM}_{e_j} fv^{\psi})  \bigr)+ c(fv^{\psi}) c(e_j) \nabla^{\cE}_{e_j} \Bigr) \\
%% Up to here %%	
	&= \sum_{j=1}^{\dim M} c(e_j)c(\nabla^{TM}_{e_j} fv^{\psi}) -2 f\nabla^{\cE}_{v^{\psi}}.
\end{split}
\]
\end{proof}
It will be important that in 
 this expression for $D_{f\psi}^2$, the only first order term, $l2\ii f\nabla^{\cE}_{v^{\psi}}$, is a $G$-differential operator.

Next, recall that we have
\[
M \cong G\times_K N
\]
as in \eqref{eq M G K}. We have a $G$-equivariant isomorphism of vector bundles
\beq{deco tangent}
TM \cong p^*T(G/K) \oplus G\times_K TN.
\eeq
This decomposition of $TM$ yields two projections
\beq{eq pGK pN}
\begin{split}
p_{G/K}\colon&TM \to p^*T(G/K);\\
p_N\colon& TM\to G\times_K TN.
\end{split}
\eeq
Identifying $T^*M\cong TM$ via the Riemannian metric as before, we obtain two partial Dirac operators
\[
\begin{split}
D_{G/K}\colon& \Gamma^{\infty}(\cE) \xrightarrow{\nabla^{\cE}} \Gamma^{\infty}(TM\otimes \cE)\xrightarrow{p_{G/K}\otimes 1_{\cE}} 
 \Gamma^{\infty}(p^*T(G/K)\otimes \cE)
\xrightarrow{c} \Gamma^{\infty}(\cE); \\
D_{N}\colon& \Gamma^{\infty}(\cE) \xrightarrow{\nabla^{\cE}} \Gamma^{\infty}(TM\otimes \cE)\xrightarrow{p_{N}\otimes 1_{\cE}} 
 \Gamma^{\infty}(G\times_K TN \otimes \cE)
\xrightarrow{c} \Gamma^{\infty}(\cE).
\end{split}
\]
Since $p_{G/K} + p_N$ is the identity map on $TM$, we have
\beq{eq decomp D}
D = D_{G/K} + D_{N}.
\eeq
This decomposition will be useful to us, because $D_{G/K}$ is a $G$-differential operator, while $D_N$ commutes with $p^*h$ for all $h\in C^{\infty}_c(G/K)$.

Combining Lemma \ref{lem Bochner} and the decomposition \eqref{eq decomp D} of $D$, we obtain an equality that we will use in our estimates. To state this equality, let $D_{G/K}^*$ and $D_N^*$ be the formal adjoints of $D_{G/K}$ and $D_N$, respectively, with respect the the $L^2$-inner product. (Note that these operators are not symmetric in general.) Consider the following operators on $\Gamma^{\infty}(\cE)$:
\beq{eq Al}
\begin{split}
A_1 &:= -D_{f\psi}D_{G/K}^*;\\
A_2 &:= -D_{G/K} D_N^*;\\
A_3 &:= -\ii fc(v^{\psi}) D_{G/K}^*;\\
A_4 &:= \ii \sum_{j=1}^{\dim M} c(e_j)c(\nabla^{TM}_{e_j} fv^{\psi});\\
A_5 &:= -2\ii f \nabla^{\cE}_{v^{\psi}},
\end{split}
\eeq
and
\beq{eq Delta}
\Delta := A_1 + A_2 + A_3 + A_4 + A_5.
\eeq
\begin{lemma} \label{lem diff inverses}
We have
\[
(D_{f\psi}^2 + 1)^{-1} - (D_N D_N^* + f^2 \|v^{\psi}\|^2 + 1)^{-1} =
(D_{f\psi}^2 + 1)^{-1} \Delta (D_N D_N^* + f^2 \|v^{\psi}\|^2 + 1)^{-1}.
\]
\end{lemma}
\begin{proof}
For any two invertible elements $a$ and $b$ of a ring, we have
\[
a^{-1} - b^{-1} = a^{-1}(b-a)b^{-1}.
\]
Hence the claim follows from Lemma \ref{lem Bochner} and \eqref{eq decomp D}.
\end{proof}

\subsection{Embeddings into complete manifolds} \label{sec cylinder}

Let $U\subset M$ be a relatively cocompact, $G$-invariant open subset, on which $\|v^{\psi}\|$ has a positive lower bound. (We will apply what follows to the sets $U_j$ in Subsection \ref{sec idea}, for $j\geq 1$.) We would like to compare the restriction of the operator $(D_{f\psi}^2+1)^{-1}$ to such sets $U$ to operators defined entirely in terms of data on $U$. But since $(D_{f\psi}^2+1)^{-1}$ is not a local operator, it does not restrict to open sets. Similarly,  
operators defined only on $U$ may not be invertible, if $U$ is not complete. For that reason, we embed $U$ into a complete manifold, in the way we will describe now.

Suppose that $\partial U$ is smooth, and that
 a $K$-invariant neighbourhood $W_N$ of $ \partial U \cap N$ in $U\cap N$ is $K$-equivariantly isometric to
\beq{eq WN}
(\partial U \cap N) \times {]{-\delta}, 0[},
\eeq
for a $\delta > 0$. Then
\[
W:= G\cdot W_N\cong G \times_K W_N
\]
is a $G$-invariant neighbourhood of $\partial U$ in $U$, and there is a $G$-equivariant isometry
\[
W\cong \partial U  \times {]{-\delta}, 0[}.
\]
By glueing
 the ``cylinder" $\partial U\times {]{-\delta}, \infty[}$ to $U$ via this identification, 
 we obtain a  manifold $V$. 
 The product of the restricted Riemannian metric from $TM$ to $T(\partial U)$ and the Euclidean metric on $T(]{-\delta}, \infty[)$ defines a $G$-invariant Riemannian metric on  $T(\partial U  \times {]{-\delta}, \infty[})$. This extends to  
 $TV$, and makes $V$ complete.
 
Define the $K$-invariant submanifold $N_V \subset V$ by attaching $(\partial U \cap N)\times {]{-\delta}, \infty[}$ to $N\cap U$ by identifying $W_N$ with \eqref{eq WN}. Then
\[
V\cong G\times_K N_V.
\]
Let the $G$-equivariant smooth map $p_V\colon V\to G/K$ be given by $p_V(gn) = gK$, for $g\in G$ and $n\in N_V$. Then $N_V = p_V^{-1}(eK)$.

The vector bundle $\cE|_U\to U$, the Clifford action by $TU$ on it, and the Dirac operator $D|_U$ all extend to $V$. See e.g.\ Section 25 of \cite{Booss93}. We denote the extended vector bundle and Dirac operator by $\cE_V$ and $D_V$, respectively.

We will extend the map $\psi|_U\colon U\to \kg$ to a map $\psi_V\colon V\to \kg$, in such a way that $\|v^{\psi_V}\|$ has a positive lower bound on $V$. 
If we only needed a continuous extension, we could use the map
 $\tilde \psi_V\colon V\to \kg$, given by
\[
\tilde \psi_V (y) = \left\{ \begin{array}{ll}
\psi(y) & \text{if $y\in U$};\\
\psi(x) & \text{if $y = (x, t) \in \partial U \times [0, \infty[$.}
\end{array} \right.
\]
The induced vector field $v^{\tilde \psi_V}$ is continuous, and its norm has the same lower bound on $V$ as $v^{\psi}$ has on $U$. To obtain a smooth version, we use the following fact. Fix any $K$-invariant norm $\|\cdot \|$ on $\kg$.
\begin{lemma}\label{lem epsilon psi}
There is an $\varepsilon > 0$, such that for all  $G$-equivariant, continuous maps $\psi'\colon \overline{U} \to \kg$ such that $\|\psi'(n) - \psi(n)\| \leq \varepsilon$ for all $n\in \overline{U}\cap N$, the norm of the vector field $v^{\psi'}$ on $\overline{U}$ has a positive lower bound.
\end{lemma}
\begin{proof}
Let $\psi'\colon \overline{U}\to \kg$ be a  $G$-equivariant, continuous map.
Since the  norm $\|v^{\psi'}\|$ is $G$-invariant, it has a positive lower bound on $\overline{U}$ precisely if it has one on $\overline{U}\cap N$. And since $\overline{U}\cap N$ is compact, $\|v^{\psi'}\|$ has a positive lower bound on this set precisely if it does not vanish there.

The set
\[
\bigl\{(n, X) \in (\overline{U}\cap N) \times \kg; X\not\in \kg_n \bigr\}
\]
is open in $(\overline{U} \cap N) \times \kg$. Since $(n, \psi(n))$ is in this set for all $n \in \overline{U}\cap N$, and $\overline{U} \cap N$ is compact, there is an $\varepsilon > 0$ such that for all $n\in \overline{U}\cap N$ and $X\in \kg$ for which $\|X-\psi(n)\|\leq \varepsilon$, we have $X^M_n \not=0$. Hence if $\|\psi'-\psi\|\leq \varepsilon$ on $\overline{U} \cap N$, the claim follows.
\end{proof}

Fix $\varepsilon > 0$ as in Lemma \ref{lem epsilon psi}. Let $W'_N \subset W_N$ be a $K$-invariant neighbourhood of $\partial U \cap N$ such that for all $n = (x, t) \in W'_N  \subset (\partial U \cap N) \times ]{-\delta}, 0[$,
\[
\|\psi(n) - \psi(x)\| \leq \varepsilon.
\]
Let $\chi \in C^{\infty}(N_V)^K$ be a function with values in $[0,1]$, such that
\[
\begin{split}
\chi \equiv 1 & \text{ on $(U\cap N)\setminus W'_N$};\\
\chi\equiv 0 & \text{ on $\partial U \times {]0, \infty[}$.}
\end{split}
\]
Define
$\psi_V\colon V\to \kg$
by
\[
\psi_V(n) = \left\{ \begin{array}{ll}
\psi(n) & \text{if $n\in (U\cap N)\setminus W_N$};\\
\chi(n) \psi(n) + (1-\chi(n))\psi(x) & \text{if $n = (x, t) \in (\partial U \cap N)\times {]{-\delta}, \infty[}$,}
\end{array}\right.
\]
and extended $G$-equivariantly to $V$. Then $\psi_V$ is smooth and $G$-equivariant, and $\|\psi'(n) - \psi(n)\| \leq \varepsilon$ for all $n\in \overline{U}\cap N$. Hence $\|v^{\psi_V}\|$ has a positive lower bound on $U$ by Lemma \ref{lem epsilon psi}. And if $(x, t)\in \partial U \times {[0, \infty[}$, then
\[
v^{\psi_V}_{(x, t)} = (v^{\psi}_x, 0) \in T_x(\partial U) \times T_t (]{-\delta}, \infty[), 
\]
and this also has a positive lower bound as $(x, t)$ ranges over $\partial U \times {[0, \infty[}$.

Let $f_V \in C^{\infty}(V)^G$ be any real-valued function such that $f_V|_U = f|_U$. Write
\[
D_{f_V \psi_V} := D_V - \ii c(v^{f_V \psi_V}).
\]
Write $D_V = D_{G/K}^V + D_{N^V}$  as in \eqref{eq decomp D}. Then Lemma \ref{lem diff inverses} applies directly to the corresponding operators on $V$.

\subsection{Operators on $M$ and $V$}

The reason for the constructions in Subsection \ref{sec cylinder} is that the manifold $V$ is complete, so that the operators
\beq{eq DV2+1}
D_{f_V \psi_V}^2 + 1
\eeq
and
\[
D_{N_V} D_{N_V}^* + f_V^2 \|v^{\psi_V}\|^2 + 1
\]
are symmetric, and norm-increasing. This implies that they are invertible, with bounded inverses with norms at most $1$. (This is generally not true for the operator $D_{f \psi}|_U^2 + 1$ on $L^2(\cE|_U)$, for example.) 
%Similarly, the operator
%\[
%D_{N_V} D_{N_V}^* + f_V^2 \|v^{\psi_V}\|^2 + 1
%\]
%is invertible as well.
%are essentially self-adjoint by Proposition \ref{prop ess sa}. 
%Since they are positive, they (or their closures) are therefore
% invertible. 
 Furthermore, after restriction to $U$, the above operators are equal to
\beq{eq DM2+1}
(D_{f \psi}^2 + 1)|_U
\eeq
and
\[
(D_{N} D_{N}^* + f^2 \|v^{\psi}\|^2 + 1)|_U,
\]
respectively. In Lemma \ref{lem est DM}, we will deduce an estimate on $U$  for the inverse of the operator $D_{f\psi}^2+1$ from the corresponding estimate for the inverse of \eqref{eq DV2+1}, using the following relation between these inverses.
%This implies that, locally on $U$, the difference between
 %the inverses of the operators \eqref{eq DV2+1} and $D_{f\psi}^2+1$ is bounded in a suitable way.
\begin{lemma} \label{lem DM DV U}
Let $\varphi_1, \varphi_2 \in \End(\cE)^G$ be supported inside $U$. Then there are $G$-equivariant bounded operators $T_0$ and $T_1$ on $L^2(\cE|_U)$, with norms at most $1$, and there is $\varphi \in \End(\cE)^G$, supported in $U$, such that for $\alpha \in \{0,1\}$,
\begin{multline} \label{eq diff DfVpsiV}
\varphi_1 D_{f\psi}^{\alpha} (D_{f\psi}^2+1)^{-1} \varphi_2 - 
\varphi_1 D_{f_V\psi_V}^{\alpha} (D_{f_V\psi_V}^2+1)^{-1} \varphi_2 \\
  = \varphi_1 T_0 \varphi  (D_{f_V\psi_V}^2+1)^{-1}\varphi_2 + 
  \varphi_1 T_1 \varphi D_{f_V\psi_V} (D_{f_V\psi_V}^2+1)^{-1}\varphi_2.
\end{multline}
\end{lemma}
\begin{proof}
Let $\chi \in C^{\infty}(U)$ be cocompactly supported, such that $\chi \equiv 1$ on $\supp \varphi_1 \cup \supp \varphi_2$. Then
\beq{eq Dfpsi loc}
\varphi_1 D_{f\psi}^{\alpha} (D_{f\psi}^2+1)^{-1} \varphi_2 = 
	\varphi_1 D_{f\psi}^{\alpha} (D_{f\psi}^2+1)^{-1} \chi (D_{f_V \psi_V}^2 + 1) (D_{f_V \psi_V}^2 + 1)^{-1}\varphi_2.
\eeq
Similarly, because $\varphi_1 D_{f_V\psi_V}^{\alpha} = \varphi_1 D_{f\psi}^{\alpha}$,
\beq{eq DfVpsiV loc}
\varphi_1 D_{f_V\psi_V}^{\alpha} (D_{f_V\psi_V}^2+1)^{-1} \varphi_2 = 
	\varphi_1 D_{f\psi}^{\alpha} (D_{f\psi}^2+1)^{-1}  (D_{f \psi}^2 + 1) 
	\chi
	(D_{f_V \psi_V}^2 + 1)^{-1}\varphi_2.
\eeq
Here we have used the fact that, while the operators $D_{f\psi}$ and $D_{f_V \psi_V}$ act on different spaces, they both act on sections of $\cE|_U$. So all compositions in \eqref{eq Dfpsi loc} and \eqref{eq DfVpsiV loc} are well-defined.
%For example, the composition $\varphi_1 D_{f_V\psi_V}^{\alpha} (D_{f\psi}^2+1)^{-1}$ in \eqref{eq DfVpsiV loc} is well-defined, even though the composition $D_{f_V\psi_V}^{\alpha} (D_{f\psi}^2+1)^{-1} $ is not.
%Because the operators $D_{f\psi}$ and $D_{f_V \psi_V}$ are in fact equal on sections of $\cE|_U$, 
Taking the difference of \eqref{eq Dfpsi loc} and \eqref{eq DfVpsiV loc}, we find that the left hand side of \eqref{eq diff DfVpsiV} equals
\beq{eq Dfpsi 2}
\varphi_1 D_{f\psi}^{\alpha} (D_{f\psi}^2+1)^{-1} 
\bigl( \chi (D_{f_V\psi_V}^2+1) - (D_{f\psi}^2+1)\chi\bigr)
(D_{f_V\psi_V}^2+1)^{-1} \varphi_2.
\eeq
Now $\chi (D_{fV\psi_V}^2+1) = \chi (D_{f\psi}^2+1)$, so
\[
\begin{split}
\chi (D_{f_V\psi_V}^2+1) - (D_{f\psi}^2+1)\chi &= [\chi, D_{f\psi}^2] \\
	&= D_{f\psi}[\chi, D_{f\psi}] +[\chi, D_{f\psi}] D_{f\psi} \\
	&=  -D_{f\psi}c(d\chi)-c(d\chi)D_{f\psi} \\
	&= -D_{f\psi}c(d\chi)-c(d\chi)D_{f_V\psi_V}. 
\end{split}
\]
Hence \eqref{eq Dfpsi 2} equals
\begin{multline*}
-\varphi_1 D_{f\psi}^{\alpha + 1} (D_{f\psi}^2+1)^{-1} 
c(d\chi)
(D_{f_V\psi_V}^2+1)^{-1} \varphi_2 \\
-
\varphi_1 D_{f\psi}^{\alpha} (D_{f\psi}^2+1)^{-1} 
c(d\chi) D_{f_V \psi_V}
(D_{f_V\psi_V}^2+1)^{-1} \varphi_2.
\end{multline*}
So the claim follows, with
\[
\begin{split}
T_{\beta} &:= D_{f\psi}^{\alpha + 1-\beta} (D_{f\psi}^2+1)^{-1} \qquad \text{(for $\beta \in \{0, 1\}$)};\\	
\varphi &:= -c(d\chi).
\end{split}
\]
\end{proof}

%%%%%%%%%%%%%%
%%% Local estimates %%%
%%%%%%%%%%%%%%

\section{Local estimates} \label{sec loc est}

After the preparations in Section \ref{sec decomp Dirac}, we are ready to prove Theorem \ref{thm def Dirac G Fred} and Proposition \ref{prop indep}.

Let $U\subset M$ be a relatively cocompact, $G$-invariant open subset of $M$, on which $\|v^{\psi}\|$ has a positive lower bound. Suppose that $U$ has the properties in Subsection \ref{sec cylinder}, i.e.\ $\partial U$ is smooth and has a tubular neighbourhood in $U$.
Let $V$, $N_V$, $p_V$, $\cE_V$, $D_V$, $f_V$ and $\psi_V$ be as in Subsection \ref{sec cylinder}. 
Suppose $f_V$ is positive.

\subsection{An estimate on $V$}

For any operator $A$ on a Hilbert space, we write $|A|:= \sqrt{A^*A}$.
\begin{lemma} \label{lem S T}
If $S$ and $T$ are operators on a Hilbert space, with $T$ positive, then
\[
\bigl| S(S^*S + T^2 + 1)^{-1} \bigr| \leq \frac{1}{2}|T^{-1}|.
\]
\end{lemma}
\begin{proof}
For any two real numbers $s$ and $t$, with $t>0$, we have
\[
\frac{|s|}{s^2 + t^2 + 1} \leq \frac{1}{2t}.
\]
This implies the claim for self-adjoint $S$. For general $S$, using a polar decomposition gives the desired estimate.
\end{proof}

The following estimate is central to our proof of Proposition \ref{prop sum G Fred}.
\begin{proposition}\label{prop est V}
For all
$e\in C^{\infty}_c(G)$ and $h\in C^{\infty}_c(G/K)$, there is a constant $B_{e, h}$, independent of $U$, and
for all $\varepsilon > 0$, there is a constant $C_{U, \varepsilon, p}$, independent of $e$ and $h$,  such that if 
\beq{eq est fV}
f_V \geq C_{U, \varepsilon, p}
\eeq
and 
\beq{eq est dfV}
f_V^2 \geq C_{U, \varepsilon, p} \|df_V\|,
\eeq
then for all $l=1, \ldots, 5$, and $\alpha \in\{0,1\}$, the operator
\beq{eq est DV}
D_{f_V\psi_V}^{\alpha} ( D_{f_V\psi_V}^2 + 1)^{-1}A_l^V \bigl( D_{N_V}D_{N_V}^* + f_V^2 \|v^{\psi_V}\|^2 + 1 \bigr)^{-1}\pi_G(e)p_V^*h
\eeq
is bounded, with norm at most $\varepsilon B_{e, h}$. Here the operators $A_l^V$ are the analogues on $V$ of the operators defined in \eqref{eq Al}.
\end{proposition}
\begin{proof}
Let $e\in C^{\infty}_c(G)$ and $h\in C^{\infty}_c(G/K)$ be given.
The operators $D_{G/K}^V$, $(D_{G/K}^V)^*$ and $\nabla^{\cE_V}_{v^{\psi_V}}$ satisfy the conditions on the operator $A$ in Proposition \ref{prop est GDO}. Indeed, they are $G$-equivariant $G$-differential operators. By construction, for these operators the vector bundle endomorphisms $a_j$ and $b$ as in Lemma \ref{lem GDO 1} are constant in the $]0, \infty[$-direction on $\partial U \times ]0, \infty[ \subset V$, and they are bounded on the compact sets $\overline{U} \cap N$.
Therefore, Proposition \ref{prop est GDO} yields a constant $B_{e, h}$, independent of $U$, and a constant $C_{U, p}$, independent of $e$ and $h$, such that the operators
\[
D_{G/K}^V \pi_G(e) p_V^*h,
\]
\[
(D_{G/K}^V)^* \pi_G(e) p_V^*h
\]
and
\[
\nabla^{\cE_V}_{v^{\psi_V}} \pi_G(e) p_V^*h
\]
are bounded, with norm at most $C_{U, p} B_{e, h}$. We choose $B_{e, h}$ so that in addition, it is at least equal to $\|e\|_{L^1(G)} \|h\|_{\infty}$.

Let $\kappa > 0$ be a lower bound for $\|v^{\psi_V}\|$. The function $\|v^{\psi_V}\|$ is $G$-invariant, hence it takes a maximum on $\overline{U}$ and therefore also on $\partial U \times ]0, \infty[$. 
Let $\|v^{\psi_V}\|_{\infty}$ be its maximum. We define the number $\|\nabla^{TM} v^{\psi_V}\|_{\infty}$ in the analogous way.
Let $\varepsilon > 0$ be given, and set
\[
C_{U, \varepsilon, p} := \max \left\{ \sqrt{\frac{C_{U, p}}{\varepsilon \kappa}}, \frac{C_{U, p}}{2\varepsilon \kappa}, \frac{C_{U, p} \|v^{\psi_V}\|_{\infty}}{\varepsilon \kappa^2}, \frac{2C_{U, p}}{\varepsilon \kappa^2}, \frac{2 \dim M}{\varepsilon \kappa}, \frac{2 \dim M  \|\nabla^{TM} v^{\psi_V}\|_{\infty}}{\varepsilon \kappa^2} \right\}.
\]
Suppose that $f_V$ satisfies \eqref{eq est fV} and \eqref{eq est dfV}.

We can then prove the desired estimates for the operators \eqref{eq est DV}, using the fact that $D_{N_V}$ commutes with $\pi_G(e)$ and $p_V^*h$. Let $s\in L^2(\cE_V)$.
\begin{enumerate}
\item First of all, we have
\begin{multline*}
\bigl\|  
D_{f_V\psi_V}^{\alpha} ( D_{f_V\psi_V}^2 + 1)^{-1}A_1^V \bigl( D_{N_V}D_{N_V}^* + f_V^2 \|v^{\psi_V}\|^2 + 1 \bigr)^{-1}\pi_G(e)p_V^*h \, s
 \bigr\|_{L^2(\cE_V)} 
\\
\leq \bigl\|  (D_{G/K}^V)^* \pi_G(e)p_V^*h \bigl( D_{N_V}D_{N_V}^* + f_V^2 \|v^{\psi_V}\|^2 + 1 \bigr)^{-1} s\bigr\|_{L^2(\cE_V)} \\
\leq C_{U, p} B_{e, h} \bigl\| \bigl(f_V^2 \|v^{\psi_V}\|^2 + 1 \bigr)^{-1} s\bigr\|_{L^2(\cE_V)} \\
\leq \varepsilon B_{e, h} \|s\|_{L^2(\cE|_V)},
 \end{multline*}
because
\[
\bigl(f_V^2 \|v^{\psi_V}\|^2 + 1 \bigr)^{-1} \leq \bigl(C_{U,p}/\varepsilon + 1 \bigr)^{-1} \leq \varepsilon/C_{U, p}.
\]
\item Next, note that
\begin{multline*}
\bigl\|  
D_{f_V\psi_V}^{\alpha} ( D_{f_V\psi_V}^2 + 1)^{-1}A_2^V \bigl( D_{N_V}D_{N_V}^* + f_V^2 \|v^{\psi_V}\|^2 + 1 \bigr)^{-1}\pi_G(e)p_V^*h \, s
 \bigr\|_{L^2(\cE_V)} 
\\
\leq \bigl\|  (D_{G/K}^V)^* \pi_G(e)p_V^*h  D_{N_V}^* \bigl( D_{N_V}D_{N_V}^* + f_V^2 \|v^{\psi_V}\|^2 + 1 \bigr)^{-1} s\bigr\|_{L^2(\cE_V)} \\
\leq C_{U, p} B_{e, h} \bigl\| D_{N_V}^* \bigl( D_{N_V}D_{N_V}^* + f_V^2 \|v^{\psi_V}\|^2 + 1 \bigr)^{-1} s\bigr\|_{L^2(\cE_V)}. \\
 \end{multline*}
 Since $C_{U, \varepsilon, p}$ is positive, the function $f_V$ does not vanish.
By Lemma \ref{lem S T}, the last expression above is at most equal to 
\[
C_{U, p} B_{e, h}\frac{1}{2} \bigl\| f_V^{-1}\|v^{\psi_V}\|^{-1} s\bigr\|_{L^2(\cE_V)} \leq \varepsilon B_{e, h} \|s\|_{L^2(\cE|_V)},
\]
 because
 \[
 f_V^{-1} \leq \frac{2\varepsilon \kappa}{C_{U, p}}.
 \]
 \item Since $(D_{G/K}^V)^*$ is a $G$-differential operator, it commutes with the $G$-invariant function $f_V$. Hence
 \begin{multline*}
\bigl\|  
D_{f_V\psi_V}^{\alpha} ( D_{f_V\psi_V}^2 + 1)^{-1}A_3^V \bigl( D_{N_V}D_{N_V}^* + f_V^2 \|v^{\psi_V}\|^2 + 1 \bigr)^{-1}\pi_G(e)p_V^*h \, s
 \bigr\|_{L^2(\cE_V)} 
\\
\leq \bigl\| c(v^{\psi_V}) (D_{G/K}^V)^* \pi_G(e)p_V^*h \, f_V \bigl( D_{N_V}D_{N_V}^* + f_V^2 \|v^{\psi_V}\|^2 + 1 \bigr)^{-1} s\bigr\|_{L^2(\cE_V)} \\
\leq C_{U, p} B_{e, h} \|v^{\psi_V}\|_{\infty}  \bigl\| f_V \bigl(  f_V^2 \|v^{\psi_V}\|^2 + 1 \bigr)^{-1} s\bigr\|_{L^2(\cE_V)} \\
\leq \varepsilon B_{e, h} \|s\|_{L^2(\cE|_V)},
 \end{multline*}
 because
\[
f_V \bigl(f_V^2 \|v^{\psi_V}\|^2 + 1 \bigr)^{-1} \leq f_V^{-1} \|v^{\psi}\|^{-2} \leq  \frac{\varepsilon}{C_{U, p} \|v^{\psi_V}\|_{\infty}}.
\]
\item For all $k = 1, \ldots, \dim M$, we have
\[
c(e_k) c(\nabla^{TM}_{e_k} f_V v^{\psi_V}) = e_k(f) c(e_k) c(v^{\psi_V}) + f_V c(e_k)c(\nabla^{TM}_{e_k} v^{\psi_V}).
\]
So we obtain the pointwise estimate
\[
\| c(e_k) c(\nabla^{TM}_{e_k} f_V v^{\psi_V}) \| \leq \|df\|  \| v^{\psi_V}\| + f_V \| \nabla^{TM} v^{\psi_V} \|.
\]
Denoting the absolute value of operators by $|\cdot |$ as before, we obtain
\begin{multline*}
\left|  \sum_{k=1}^{\dim M}c(e_k) c(\nabla^{TM}_{e_k} f_V v^{\psi_V})  \bigl( D_{N_V}D_{N_V}^* + f_V^2 \|v^{\psi_V}\|^2 + 1 \bigr)^{-1}  \right| \\
\leq \dim M \bigl( \|df_V\|  \| v^{\psi_V}\| + f_V \| \nabla^{TM} v^{\psi_V} \| \bigr) \bigl( f_V^2 \|v^{\psi_V}\|^2 + 1 \bigr)^{-1}. 
\end{multline*}
Therefore,
 \begin{multline*}
\bigl\|  
 D_{f_V\psi_V}^{\alpha} ( D_{f_V\psi_V}^2 + 1)^{-1}A_4^V \bigl( D_{N_V}D_{N_V}^* + f_V^2 \|v^{\psi_V}\|^2 + 1 \bigr)^{-1}\pi_G(e)p_V^*h \, s
	  \bigr\|_{L^2(\cE_V)} \\
 \leq\dim M \left\| \bigl( \|df_V\|  \| v^{\psi_V}\| + f_V \| \nabla^{TM} v^{\psi_V} \|  \bigr)\bigl( f_V^2 \|v^{\psi_V}\|^2 + 1 \bigr)^{-1} \right\|  \| e\|_{L^1(G)} \|h\|_{\infty} \| s \|_{L^2(\cE_V)}\\
\leq \varepsilon B_{e, h} \|s\|_{L^2(\cE_{V})},
 \end{multline*}
because
\[
\frac{\|df\| \|v^{\psi_V}\|}{f_V^2 \|v^{\psi_V}\|^2 + 1} \leq \frac{\varepsilon}{2 \dim M},
\]
and
\[
\frac{f_V  \| \nabla^{TM} v^{\psi_V} \|_{\infty} }{f_V^2 \|v^{\psi_V}\|^2 + 1} \leq \frac{\varepsilon}{2 \dim M}. 
\]
\item Finally, since $\nabla^{\cE_V}_{v^{\psi_V}}$ is a $G$-differential operator, it commutes with the $G$-invariant function $f_V$. Hence
 \begin{multline*}
\bigl\|  
 D_{f_V\psi_V}^{\alpha} ( D_{f_V\psi_V}^2 + 1)^{-1}A_5^V \bigl( D_{N_V}D_{N_V}^* + f_V^2 \|v^{\psi_V}\|^2 + 1 \bigr)^{-1}\pi_G(e)p_V^*h \, s
	 \bigr\|_{L^2(\cE_V)} \\
2 \| \nabla^{\cE}_{\psi_V} \pi_G(e)p_V^*h \, f_V \bigl( D_{N_V}D_{N_V}^* + f_V^2 \|v^{\psi_V}\|^2 + 1 \bigr)^{-1}\, s
	 \bigr\|_{L^2(\cE_V)} \\
\leq 2 C_{U, p} B_{e, h} \bigl\| \bigl( f_V^2 \|v^{\psi_V}\|^2 + 1 \bigr)^{-1}\, s
	 \bigr\|_{L^2(\cE_V)}\\
	 \leq
 \varepsilon B_{e, h} \|s\|_{L^2(\cE_{V})},
 \end{multline*}
 because
 \[
 \bigl( f_V^2 \|v^{\psi_V}\|^2 + 1 \bigr)^{-1} \leq \frac{\varepsilon}{2 C_{U, p}}.
 \]
\end{enumerate}
\end{proof}

\subsection{An estimate for $D_{N_V}$}

Proposition \ref{prop est V} will allow us to deduce estimates for $(D_{f\psi}^2+1)^{-1}$ from estimates for $\bigl( D_{N_V}D_{N_V}^* + f_V^2 \|v^{\psi_V}\|^2 + 1 \bigr)^{-1}$.
\begin{lemma} \label{lem est DNV}
For all $e\in C^{\infty}_c(G)$ and $h\in C^{\infty}_c(G/K)$, there is a constant $B_{e, h}$, and
for all $\varepsilon > 0$, there is a constant $C_{U, \varepsilon, p} > 0$, independent of $e$ and $h$, such  such that if 
\[
f_V \geq C_{U, \varepsilon, p}
\]
then for $\alpha \in\{0,1\}$, the operator
\beq{eq est DNV}
D_{f_V\psi_V}^{\alpha} \bigl( D_{N_V}D_{N_V}^* + f_V^2 \|v^{\psi_V}\|^2 + 1 \bigr)^{-1}\pi_G(e)p_V^*h
\eeq
is bounded, with norm at most $\varepsilon B_{e, h}$.
\end{lemma}
\begin{proof}
Let $e\in C^{\infty}_c(G)$ and $h\in C^{\infty}_c(G/K)$ be given. Let $B_{e, h}$ be as in Proposition \ref{prop est GDO}, but also at least equal to $\|e\|_{L^1(G)} \|h\|_{\infty}$. 
Let $C_{D_{G/K}^V, p}$ be as in Proposition \ref{prop est GDO}, for $A = D_{G/K}^V$. Let $\kappa > 0$ be a lower bound for $\|v^{\psi_V}\|$.
Let $\varepsilon > 0$ be given, and set
\[
C_{U, \varepsilon, p} := \max\left\{  \frac{1}{\sqrt{\varepsilon} \kappa}, \frac{3}{\varepsilon \kappa}, \sqrt{\frac{3 C_{D_{G/K}^V, p}}{\varepsilon \kappa^2}} \right\}.
\]
Suppose $f_V \geq C_{U, \varepsilon, p}$.

For $\alpha = 0$, the operator \eqref{eq est DNV} has norm at most
\[
\left\| \bigl( f_V^2 \|v^{\psi_V}\|^2 + 1 \bigr)^{-1} \right\|  \|e\|_{L^1(G)} \|h\|_{\infty} \leq \varepsilon B_{e, h},
\]
because
$
\bigl(f_V^2 \|v^{\psi_V}\|^2 + 1 \bigr)^{-1} \leq \varepsilon.
$

Now suppose $\alpha = 1$. Write
\[
D_{f_V \psi_V} = D_{N_V} + D_{G/K}^V -\ii f_V c(v^{\psi_V}).
\]
By Lemma \ref{lem S T}, we have
\begin{multline*}
\bigl\| D_{N_V} \bigl( D_{N_V}D_{N_V}^* + f_V^2 \|v^{\psi_V}\|^2 + 1 \bigr)^{-1}\pi_G(e)p_V^*h \bigr\| \\
\leq\frac{1}{2} \bigl\| f_V^{-1} \|v^{\psi_V}\|^{-1} \bigr\| \|e\|_{L^1(G)} \|h\|_{\infty} %\\
\leq \varepsilon B_{e, h}/3,
\end{multline*}
because
\[
 f_V^{-1} \|v^{\psi_V}\|^{-1} \leq 2\varepsilon /3.
\]

Next, because $\pi_G(e)$ and $p_V^*h$ commute with $D_{N_V}$, we have
\begin{multline*}
\bigl\| D_{G/K}^V  \bigl( D_{N_V}D_{N_V}^* + f_V^2 \|v^{\psi_V}\|^2 + 1 \bigr)^{-1}\pi_G(e)p_V^*h \bigr\| \\
	\leq C_{D_{D/K}^V, p} B_{e, h} \bigl\|  \bigl( f_V^2 \|v^{\psi_V}\|^2 + 1 \bigr)^{-1} \bigr\| % \\
	\leq \varepsilon B_{e, h}/3,
\end{multline*}
because
\[
 ( f_V^2 \|v^{\psi_V}\|^2 + 1)^{-1} \leq \frac{\varepsilon}{3 C_{D_{G/K}^V, p}}.
\]

Finally,
\begin{multline*}
\bigl\| \ii f_V c(v^{\psi_V}) \bigl( D_{N_V}D_{N_V}^* + f_V^2 \|v^{\psi_V}\|^2 + 1 \bigr)^{-1}\pi_G(e)p_V^*h \bigr\| \\
\leq \bigl\| f_V \|v^{\psi_V}\|  \bigl( f_V^2 \|v^{\psi_V}\|^2 + 1 \bigr)^{-1} \bigr\| \|e\|_{L^1(G)} \|h\|_{\infty} %\\
\leq \varepsilon B_{e, h}/3,
\end{multline*}
because
$
 f_V \|v^{\psi_V}\|  \bigl( f_V^2 \|v^{\psi_V}\|^2 + 1 \bigr)^{-1} \leq \varepsilon/3.
$
\end{proof}

\subsection{Proof of Proposition \ref{prop sum G Fred}}

By combining Lemma \ref{lem DM DV U}, Proposition \ref{prop est V} and Lemma \ref{lem est DNV}, we obtain an estimate for the inverse of $D_{f\psi}^2+1$ on $U$. This begins with an estimate for the inverse of $D_{f_V\psi_V}^2+1$.
\begin{lemma} \label{lem est DV}
For all $e\in C^{\infty}_c(G)$ and $h\in C^{\infty}_c(G/K)$, there is a constant $B_{e, h}$, and
for all $\varepsilon > 0$, and all $\varphi_1, \varphi_2 \in \End(E)^G$, supported in $U$, there is a constant $C_{\varphi_1, \varphi_2, \varepsilon, p} > 0$, independent of $e$ and $h$,
 such that if 
\beq{eq f U}
f_V \geq C_{\varphi_1, \varphi_2, \varepsilon, p}
\eeq
and
\beq{eq df U}
f_V^2 \geq
  C_{\varphi_1, \varphi_2, \varepsilon, p} \|df_V\|,
\eeq
then for $\alpha \in\{0,1\}$, we have
\beq{eq norm Dfpsi U}
\bigl\| \varphi_1  D_{f_V\psi_V}^{\alpha} (D_{f_V\psi_V}^2+1)^{-1} \pi_G(e) p^*h \varphi_2 \bigr\| \leq \varepsilon B_{e, h}.
\eeq
\end{lemma}
\begin{proof}
Let $e\in C^{\infty}_c(G)$ and $h\in C^{\infty}_c(G/K)$ be given. Let $B_{e, h}$ be at least as large as the constants $B_{e, h}$ in Proposition \ref{prop est V} and Lemma \ref{lem est DNV}.
%, and also at least equal to $\|e\|_{L^1(G)} \|h\|_{\infty}$. 
Let $\varepsilon > 0$ be given.
Because $\varphi_1$ and $\varphi_2$ are $G$-equivariant and supported in $U$, they are bounded operators on $L^2(\cE)$. 
Set
\[
\varepsilon' := \frac{\varepsilon}{6(\|\varphi_1\|\,  \|\varphi_2\| + 1)}
\]
Let $C_{\varphi_1, \varphi_2, \varepsilon, p}$ be the maximum of the constants
 $C_{U, \varepsilon', p}$ in Proposition \ref{prop est V} and 
Lemma \ref{lem est DNV}.
%, for this value of $\varepsilon'$.
% Let $\kappa > 0$ be a lower bound for $\|v^{\psi}\|$ on $U$. 
%Set
%\[
%C_{\varphi_1, \varphi_2, \varepsilon, p} := \max\left\{ C_{U, \varepsilon', p}, \tilde C_{U, \varepsilon', p} \right\}.
%\]

Suppose that $f_V$ satisfies \eqref{eq f U} and \eqref{eq df U}. 
%Then we can choose a positive function $f_V \in C^{\infty}(V)^G$ such that $f_V|_U = f|_U$ and $f_V$ satisfies 
%\[
%f_V \geq C_{\varphi_1, \varphi_2, \varepsilon, p}/2
%\]
%and
%\[
%f_V^2 \geq
%  C_{\varphi_1, \varphi_2, \varepsilon, p} \|df_V\|/2.
%\]
% By Lemma \ref{lem DM DV U}, the norm on the left hand side of \eqref{eq norm Dfpsi U} equals
%\[
%\bigl\| \varphi_1  D_{f_V\psi_V}^{\alpha} (D_{f_V\psi_V}^2+1)^{-1} \pi_G(e) p^*h \varphi_2 \bigr\|. 
%\]
By Lemma \ref{lem diff inverses}, the norm on the left hand side of \eqref{eq norm Dfpsi U}  is at most equal to
\begin{multline} \label{eq norm D}
\bigl\| \varphi_1  D_{f_V\psi_V}^{\alpha} (D_{N_V}D_{N_V}^* + f_V^2 \|v^{\psi_V}\|^2+1)^{-1} \pi_G(e) p^*h \varphi_2 \bigr\| \\
+ 
\bigl\| \varphi_1 D_{f_V\psi_V}^{\alpha} (D_{f_V\psi_V}^2+1)^{-1}  \Delta  
\bigl(D_{N_V}D_{N_V}^* + f_V^2 \|v^{\psi_V}\|^2 + 1\bigr)^{-1}
%(D_{f_V\psi_V}^2+1)^{-1} 
\pi_G(e) p^*h \varphi_2 \bigr\|. 
\end{multline}
Here $\Delta$ was defined in \eqref{eq Delta}.
Lemma \ref{lem est DNV} implies that the first  term in \eqref{eq norm D} is at most equal to
\[
\|\varphi_1\| \|\varphi_2\| \varepsilon' B_{e, h} \leq \varepsilon B_{e, h}/6.
\]
Proposition \ref{prop est V} implies that the second term in \eqref{eq norm D} is at most equal to
\[
5\|\varphi_1\| \, \|\varphi_2\| \varepsilon' B_{e, h} \leq 5\varepsilon B_{e, h}/6,
\]
so the claim follows.
\end{proof}

Using Lemma \ref{lem est DV}, we obtain the estimate for the inverse of $D_{f\psi}^2+1$ that we need.
\begin{lemma} \label{lem est DM}
For all $e\in C^{\infty}_c(G)$ and $h\in C^{\infty}_c(G/K)$, there is a constant $B_{e, h}$, and
for all $\varepsilon > 0$, and all $\varphi_1, \varphi_2 \in \End(E)^G$, supported in $U$, there is a constant $C_{\varphi_1, \varphi_2, \varepsilon, p} > 0$, independent of $e$ and $h$,
 such that if 
\beq{eq f U M}
f|_U \geq C_{\varphi_1, \varphi_2, \varepsilon, p}
\eeq
and
\beq{eq df U M}
f^2|_U \geq
  C_{\varphi_1, \varphi_2, \varepsilon, p} \|df|_U\|,
\eeq
then for $\alpha \in\{0,1\}$, we have
\beq{eq norm DfpsiM U}
\bigl\| \varphi_1  D_{f\psi}^{\alpha} (D_{f\psi}^2+1)^{-1} \pi_G(e) p^*h \varphi_2 \bigr\| \leq \varepsilon B_{e, h}.
\eeq
\end{lemma}
\begin{proof}
Let $e \in C^{\infty}_c(G)$ and $h \in C^{\infty}_c(G/K)$ be given. Let $B_{e, h}$ be as in Lemma \ref{lem est DV}. Let $\varphi_1, \varphi_2 \in \End(E)^G$ be supported in $U$, and let $\varepsilon > 0$. Let $T_0$,  $T_1$ and $\varphi$ be as in Lemma \ref{lem DM DV U}.
We use tildes on the constants in Lemma \ref{lem est DV} to distinguish them from the constants in this lemma, and set
\[
\begin{split}
\varepsilon' &:= \varepsilon/(\|\varphi_1\|+1);\\
C_{\varphi_1, \varphi_2, \varepsilon, p} &:= 2 \max\{\tilde C_{\varphi_1, \varphi_2, \varepsilon/3, p},
	\tilde C_{\varphi, \varphi_2, \varepsilon'/3, p} \}.
\end{split}
\]
By Lemma \ref{lem DM DV U}, the norm on the left hand side of \eqref{eq norm DfpsiM U} is at most equal to
\begin{multline}\label{eq est DV 3 terms}
\bigl\| \varphi_1  D_{f_V\psi_V}^{\alpha} (D_{f_V\psi_V}^2+1)^{-1} \pi_G(e) p^*h \varphi_2 \bigr\| 
\\
+ 
\bigl\| \varphi_1 T_0 \| \cdot \| \varphi (D_{f_V\psi_V}^2+1)^{-1} \pi_G(e) p^*h \varphi_2 \bigr\|
\\
+ 
\bigl\| \varphi_1 T_1 \| \cdot \| \varphi D_{f_V\psi_V} (D_{f_V\psi_V}^2+1)^{-1} \pi_G(e) p^*h \varphi_2 \bigr\|.
\end{multline}
Suppose that $f$ satisfies \eqref{eq f U M} and \eqref{eq df U M}. Then $f_V$ can be chosen so that it satisfies
satisfies \eqref{eq f U} and \eqref{eq df U}, with $C_{\varphi_1, \varphi_2, \varepsilon, p}$ in \eqref{eq f U} and \eqref{eq df U} replaced by $C_{\varphi_1, \varphi_2, \varepsilon, p}/2$ as chosen in this proof. 
Then
 Lemma \ref{lem est DV} implies that \eqref{eq est DV 3 terms} is at most equal to 
\[
\frac{\varepsilon}{3} \Bigl( 1 + 2\frac{\| \varphi_1 \|}{\|\varphi_1\|+1} \Bigr)B_{e, h} \leq \varepsilon B_{e, h}.
\]
\end{proof}

We are now ready to prove Proposition \ref{prop sum G Fred}, and hence Theorem \ref{thm def Dirac G Fred}.

\medskip\noindent
\emph{Proof of Proposition \ref{prop sum G Fred}.}
Let an open cover $\{U_j\}_{j=1}^{\infty}$ of $M$ and a partition of unity $\{\chi_j^2 \}_{j=0}^{\infty}$ as in Subsection \ref{sec idea} be given. For each $j\geq 1$, consider the vector bundle endomorphisms $\chi_j$ and $c(d\chi_j)$
%\[
%\begin{split}
%\varphi_{1, j} &:= c(d\chi_j);\\
%\varphi_{2, j} &:= \chi_j
%\end{split}
%\]
of $\cE$. For each $j\geq 1$, set
\[
C_j := \max\{ C_{\chi_j, \chi_j, 2^{-j}/3}, C_{c(d\chi_j), \chi_j, 2^{-j}/3} \},
\]
with $C_{\chi_j, \chi_j, 2^{-j}/3}$  and $C_{c(d\chi_j), \chi_j, 2^{-j}/3}$ as in Lemma \ref{lem est DM}. Because every point in $m$ lies in only finitely many of the sets $U_j$, there is a function $\rho \in C^{\infty}(M)^G$ such that for all $j\geq 1$, 
\[
\rho|_{U_j} \geq C_j.
\]
Suppose that $f\in C^{\infty}(M)^G$ is $\rho$-admissible. Then for all $j$,
\[
\begin{split}
f|_{U_j} &\geq C_j;\\
f|_{U_j}^2 &\geq C_j \|df|_{U_j}\|. 
\end{split}
\]
Let $e \in C^{\infty}_c(G)$ and $h\in C^{\infty}_c(G/K)$ be given, and let $B_{e, h}$ be as in Lemma \ref{lem est DM}. 

Note that for all $j$,
\[
(D_{f\psi}^2+1)^{-1} \pi_G(e) p^*h \,  \chi_j^2 =\chi_j (D_{f\psi}^2+1)^{-1} \pi_G(e) p^*h \,  \chi_j + \bigl[(D_{f\psi}^2+1)^{-1}, \chi_j \bigr] \pi_G(e) p^*h \,  \chi_j. 
\]
Now
\[
\bigl[(D_{f\psi}^2+1)^{-1}, \chi_j \bigr] = (D_{f\psi}^2+1)^{-1}  \bigl( D_{f\psi} c(d\chi_j) + c(d\chi_j)D_{f\psi} \bigr) (D_{f\psi}^2+1)^{-1}. 
\]
Therefore,
\begin{multline*}
\bigl\| (D_{f\psi}^2+1)^{-1} \pi_G(e) p^*h \,  \chi_j^2 \bigr\| \leq \bigl\| \chi_j (D_{f\psi}^2+1)^{-1} \pi_G(e) p^*h \,  \chi_j \bigr\| \\
	+ \bigl\| (D_{f\psi}^2+1)^{-1}  D_{f\psi} c(d\chi_j)  (D_{f\psi}^2+1)^{-1}\pi_G(e) p^*h \,  \chi_j \bigr\| \\
	+ \bigl\| (D_{f\psi}^2+1)^{-1}  c(d\chi_j)  D_{f\psi}  (D_{f\psi}^2+1)^{-1} \pi_G(e) p^*h \,  \chi_j\bigr\|. 
\end{multline*}
By Lemma \ref{lem est DM}, all three terms on the right hand side are at most equal to $2^{-j}B_{e, h}/3$ for all $j\geq 1$. Hence Proposition \ref{prop sum G Fred} follows. 
\hfill $\square$

\medskip
By the arguments in Subsection \ref{sec idea}, Proposition \ref{prop sum G Fred} implies Theorem \ref{thm def Dirac G Fred}.

\subsection{Independence of choices}

Let us prove Proposition \ref{prop indep}. We start by showing that different admissible functions lead to the same index.
\begin{lemma}\label{lem indep f}
In the setting of Theorem \ref{thm def Dirac G Fred}, let $f_0$ and $f_1$ be two $\rho$-admissible functions. Then
\[
\ind_G D_{f_0 \psi} = \ind_G D_{f_1 \psi}.
\]
\end{lemma}
\begin{proof}
Let $f_0$ and $f_1$ be two $\rho$-admissible functions. Then for all constants $a\geq 1$ and $j=0,1$, the function $af_j$ is $\rho$-admissible. Furthermore, by a homotopy argument, we have
\[
\ind_G (D_{af_j\psi}) = \ind_G(D_{f_j\psi}) 
\]

For $t\in [0,1]$, we consider 
the function 
\[
h_t := 4\bigl( (1-t)f_0 + tf_1 \bigr)\quad \in C^{\infty}(M)^G.
\]
For all $t\in [0,1]$, we have
\[
\begin{split}
\frac{h_t^2}{\|dh_t\|+h_t + 1} &= \frac{h_t^2}{4\bigl( (1-t)(\|df_0\| + f_0 + 1/4) + t(\|df_1\|+f_1 + 1/4)\bigr)} \\
&\geq \frac{4\bigl( (1-t)^2f_0^2 + t^2 f_1^2 \bigr)}{ (1-t)(\|df_0\| + f_0 + 1) + t(\|df_1\|+f_1 + 1)} \\
&\geq \frac{2\bigl( (1-t)^2f_0^2 + t^2 f_1^2 \bigr)}{ \max_{j=0,1}(\|df_j\| + f_j + 1) } \\
&\geq 2 \bigl( (1-t)^2\rho + t^2\rho \bigr)\\
&\geq \rho,
\end{split}
\]
so $h_t$ is $\rho$-admissible for all $t$.

We conclude that, by  operator homotopies,
\[
\begin{split}
\ind_G D_{f_0 \psi} &= \ind_G(D_{h_0\psi}) \\
	&= \ind_G(D_{h_1\psi}) \\
&=\ind_G D_{f_1 \psi}.
\end{split}
\]
\end{proof}

The space of  $L^2$-sections of $\cE$ depends on the Riemannian metric, through the Riemannian density. Therefore, operator homotopies, for operators on a fixed Hilbert space, are not enough to prove the remaining part of Proposition \ref{prop indep}. We can use an argument
 modelled on Section 11.2 of \cite{higson00}, however.  We use the notation of Proposition \ref{prop indep}.

Let $B_{S^1}$ be the standard Riemannian metric on $TS^1$. Let $\C_1$ be the complex Clifford algebra with one generator $e_1$. Then, as a complex vector space, $\C_1 = \Span_{\C}\{1, e_1\} \cong \C^2$. Consider the spinor bundle $\cE_{S^1} = S^1 \times \C_1\to S^1$. We have the Dirac operator $D_{S^1} = c(e_1)\frac{d}{d\alpha}$, where $\alpha$ is the angle coordinate on $S^1$. If $I\subset {]0, 2\pi[}$ is an open sub-interval, we embed it into $S^1$ via the map $x\mapsto e^{ix}$. We write $B_I$, $\cE_I$ and $D_I$ for the restrictions to $I$ of $B_{S^1}$, $\cE_{S^1}$ and $D_{S^1}$, respectively. We set $I_0 := {]0, \pi/2[}$ and $I_1 := {]\pi, 3\pi/2[}$.

\begin{lemma}\label{lem D tilde}
There is a Riemannian metric $\tilde B$ on $S^1 \times M$, a $G$-equivariant Clifford module $\tilde \cE \to S^1\times M$, a $G$-invariant Clifford connection $\nabla^{\tilde \cE}$ on $\tilde \cE$, and a smooth, $G$-equivariant map $\tilde p\colon S^1\times M\to G/K$, such that for $j=0,1$, the metric $\tilde B|_{I_j\times M}$ is the product metric of $B_{I_j}$ and $B_j$, we have
\[
\tilde \cE|_{I_j \times M} \cong \cE_{I_j} \boxtimes \cE_j \to I_j \times M
\]
as $G$-equivariant Clifford modules, the Dirac operator $\tilde D$ associated to $\nabla^{\tilde \cE}$ satisfies
\[
\tilde D|_{I_j\times M} = D_{I_j}\otimes 1 + 1\otimes D^{\cE_j},
\]
and for all $t\in I_j$ and $m\in M$, we have
\[
\tilde p(t, m) = p_j(m).
\]
\end{lemma}
\begin{proof}
As a $G$-equivariant vector bundle, we take $\tilde \cE = \cE_{S^1} \boxtimes \cE$.
The metric $\tilde B$, Clifford action on $\tilde \cE$ and Clifford connection on $\tilde \cE$ can be constructed using a partition of unity. The map $\tilde p$ exists because $G/K$ is $G$-equivariantly contractible.
\end{proof}

Fix $\tilde B$, $\tilde \cE$, $\nabla^{\tilde \cE}$, $\tilde D$ and $\tilde p$ as in Lemma \ref{lem D tilde}. Let $\tilde \psi\colon S^1\times M\to \kg$ and $\tilde f \in C^{\infty}(S^1 \times M)^G$ be the pullbacks of $\psi$ and $f$, respectively. Consider the deformed Dirac operator
\[
\tilde D_{\tilde f\tilde \psi} = \tilde D - \ii \tilde f \tilde c(v^{\tilde \psi}),
\]
with $\tilde c$ the Clifford action by $T(S^1 \times M)$ on $\tilde \cE$.
\begin{lemma}\label{lem tilde D KK}
There is a positive function $\tilde \rho \in C^{\infty}(S^1\times M)^G$ such that if $\tilde f$ is $\tilde \rho$-admissible, the triple
\beq{eq tilde D KK}
\Bigl(L^2(\tilde \cE), \frac{\tilde D_{\tilde f\tilde \psi}}{\sqrt{\tilde D_{\tilde f\tilde \psi}^2 + 1}}, \pi_{S^1} \otimes \pi_{G, G/K} \Bigr)
\eeq
is a Kasparov $\bigl(C(S^1) \otimes C_0(G/K)\rtimes G, \C \bigr)$-cycle. Here $\pi_{S_1}\colon C(S^1) \to \cB(L^2(\tilde \cE))$ is defined by pointwise multiplication after pulling back to $S^1 \times M$.
\end{lemma}
\begin{proof}
Since $S^1$ is compact, the set of zeroes of the vector field $v^{\tilde \psi}$ is cocompact. Hence Theorem \ref{thm def Dirac G Fred} implies that there is a function $\tilde \rho$ such that if $\tilde f$ is $\tilde \rho$-admissible, the triple \eqref{eq tilde D KK} is a Kasparov $(C_0(G/K)\rtimes G, \C)$-cycle. Since $\pi_{S^1}$ commutes with $\pi_{G, G/K}$, the claim follows.
\end{proof}
\begin{remark}
Lemma \ref{lem tilde D KK} still holds if $X$ is replaced by any compact manifold $X$. The authors also expect it to be true if $X$ is noncompact but complete. Then the representation $\pi_{X}$ (of $C_0(X)$ in that case) plays a more important role.
\end{remark}

Let $\tilde \rho$ be as in Lemma \ref{lem tilde D KK}, and suppose $\tilde f$ is $\tilde \rho$-admissible. Let 
\[
[\tilde D_{\tilde f\tilde \psi}] \in KK_1(C(S^1) \otimes C_0(G/K)\rtimes G, \C)
\]
be the class defined by \eqref{eq tilde D KK}. Note that this is a class in odd $KK$-theory. For an open interval $I\subset {]0, 2\pi[}$, consider the restriction map
\[
r_I\colon KK_1(C(S^1) \otimes C_0(G/K)\rtimes G, \C)\to KK_1(C_0(I) \otimes C_0(G/K)\rtimes G, \C),
\]
induced by the inclusion map $C_0(I)\hookrightarrow C(S^1)$. 
%(We will use the same notation if $S^1$ is replaced by $I$, and $I$ by an open sub-interval $I'\subset I$.) 
We also have the suspension isomorphism
\[
s_I\colon KK_1(C_0(I) \otimes C_0(G/K)\rtimes G, \C)\xrightarrow{\cong} KK_0(C_0(G/K)\rtimes G, \C),
\]
see Definition 9.5.6 in \cite{higson00}. The core of the proof of Proposition \ref{prop indep} is the fact that we can recover $G$-indices of deformed Dirac operators from operators on $S^1\times M$, using the suspension isomorphism. 
\begin{lemma}\label{lem suspension}
For all open intervals $I\subset {]0, 2\pi[}$, and $j=0,1$, we have
\[
s_I[D_I \otimes 1 + 1\otimes D^{\cE_j}_{f_j\psi}] = \ind_G^{p^j}(D^{\cE_j}_{f_j\psi}) \quad \in KK_0(C_0(G/K)\rtimes G, \C).
\]
\end{lemma}
\begin{proof}
As in Lemmas 9.5.7 and 9.5.8, Exercise 10.9.7 and Proposition 11.2.5 in \cite{higson00}, we have
\[
\begin{split}
s_I[D_I \otimes 1 + 1\otimes D^{\cE_j}_{f_j\psi}] &= s_I[D_I] \otimes \ind_G^{p^j}(D^{\cE_j}_{f_j\psi}) \\
&=\ind_G^{p^j}(D^{\cE_j}_{f_j\psi}). 
\end{split}
\]
\end{proof}

\noindent \emph{Proof of Proposition \ref{prop indep}.}
Since $S^1$ is compact, we can choose $f$ so that $\tilde f$ is $\tilde \rho$-admissible. Furthermore, we can choose $f$ so that in addition, it is $\max(\rho_0, \rho_1)$-admissible. Then by Lemma \ref{lem indep f}, we have for $j=0,1$,
\beq{eq fj f}
\ind_G^{p_j}(D^{\cE_j}_{f_j\psi}) = \ind_G^{p_j}(D^{\cE_j}_{f\psi}).
\eeq

Set $I:= {]0, 2\pi[}$. Consider the class
\[
r_I [\tilde D_{\tilde f\tilde \psi}] \in KK_1(C_0(I) \otimes C_0(G/K)\rtimes G, \C).
\]
The inclusion maps $I_j \hookrightarrow I$ are homotopy equivalences. Therefore, after identifying $I_j \cong I$, we have
\[
\begin{split}
\bigl[D_I \otimes 1 + 1\otimes D^{\cE_0}_{f\psi} \bigr] &= r_{I_0}  [\tilde D_{\tilde f\tilde \psi}] \\
&= r_I [\tilde D_{\tilde f\tilde \psi}] \\
&= r_{I_1} [\tilde D_{\tilde f\tilde \psi}] \\
&=\bigl[D_I \otimes 1 + 1\otimes D^{\cE_1}_{f\psi}\bigr] \\
&\qquad \in KK_1(C_0(I)\otimes C_0(G/K)\rtimes G, \C). 
\end{split}
\]
Hence, by Lemma \ref{lem suspension}, 
\[
\ind_G^{p_0}(D^{\cE_0}_{f\psi})  = \ind_G^{p_1}(D^{\cE_1}_{f\psi}). 
\]
Combined with \eqref{eq fj f}, this implies that
\[
\ind_G^{p_0}(D^{\cE_0}_{f_0\psi}) = \ind_G^{p_1}(D^{\cE_1}_{f_1\psi}). 
\]
\hfill $\square$

\begin{small}

\bibliographystyle{plain}
\bibliography{mybib}

\begin{thebibliography}{10}

\bibitem{Abels}
Herbert Abels.
\newblock Parallelizability of proper actions, global {$K$}-slices and maximal
  compact subgroups.
\newblock {\em Math. Ann.}, 212:1--19, 1974/75.

\bibitem{Anghel93}
Nicolae Anghel.
\newblock On the index of {C}allias-type operators.
\newblock {\em Geom. Funct. Anal.}, 3(5):431--438, 1993.

\bibitem{Atiyah74}
Michael Atiyah.
\newblock {\em Elliptic operators and compact groups}.
\newblock Lecture Notes in Mathematics, Vol. 401. Springer-Verlag, Berlin,
  1974.

\bibitem{Atiyah77}
Michael Atiyah and Wilfried Schmid.
\newblock A geometric construction of the discrete series for semisimple {L}ie
  groups.
\newblock {\em Invent. Math.}, 42:1--62, 1977.

\bibitem{Baaj83}
Saad Baaj and Pierre Julg.
\newblock Th\'eorie bivariante de {K}asparov et op\'erateurs non born\'es dans
  les {$C^{\ast} $}-modules hilbertiens.
\newblock {\em C. R. Acad. Sci. Paris S\'er. I Math.}, 296(21):875--878, 1983.

\bibitem{Connes94}
Paul Baum, Alain Connes, and Nigel Higson.
\newblock Classifying space for proper actions and {$K$}-theory of group
  {$C^\ast$}-algebras.
\newblock In {\em {$C^\ast$}-algebras: 1943--1993 ({S}an {A}ntonio, {TX},
  1993)}, volume 167 of {\em Contemp. Math.}, pages 240--291. American
  Mathematical Society, Providence, RI, 1994.

\bibitem{Blackadar}
Bruce Blackadar.
\newblock {\em {$K$}-theory for operator algebras}, volume~5 of {\em
  Mathematical Sciences Research Institute Publications}.
\newblock Cambridge University Press, Cambridge, second edition, 1998.

\bibitem{Booss93}
Bernhelm Boo{\ss}-Bavnbek and Krzysztof~P. Wojciechowski.
\newblock {\em Elliptic boundary problems for {D}irac operators}.
\newblock Mathematics: Theory \& Applications. Birkh\"auser Boston, Inc.,
  Boston, MA, 1993.

\bibitem{Braverman02}
Maxim Braverman.
\newblock Index theorem for equivariant {D}irac operators on noncompact
  manifolds.
\newblock {\em $K$-Theory}, 27(1):61--101, 2002.

\bibitem{Braverman14}
Maxim Braverman.
\newblock The index theory on non-compact manifolds with proper group action.
\newblock {\em J. Geom. Phys.}, 98:275--284, 2015.

\bibitem{Bruening90}
Jochen Br{\"u}ning.
\newblock {$L^2$}-index theorems on certain complete manifolds.
\newblock {\em J. Differential Geom.}, 32(2):491--532, 1990.

\bibitem{Bruening92a}
Jochen Br{\"u}ning.
\newblock On {$L^2$}-index theorems for complete manifolds of rank-one-type.
\newblock {\em Duke Math. J.}, 66(2):257--309, 1992.

\bibitem{Bruening92b}
Jochen Br{\"u}ning and Henri Moscovici.
\newblock {$L^2$}-index for certain {D}irac-{S}chr\"odinger operators.
\newblock {\em Duke Math. J.}, 66(2):311--336, 1992.

\bibitem{Bunke95}
Ulrich Bunke.
\newblock A {$K$}-theoretic relative index theorem and {C}allias-type {D}irac
  operators.
\newblock {\em Math. Ann.}, 303(2):241--279, 1995.

\bibitem{Callias78}
Constantine Callias.
\newblock Axial anomalies and index theorems on open spaces.
\newblock {\em Comm. Math. Phys.}, 62(3):213--234, 1978.

\bibitem{Connes82}
Alain Connes and Henri Moscovici.
\newblock The {$L^{2}$}-index theorem for homogeneous spaces of {L}ie groups.
\newblock {\em Ann. of Math. (2)}, 115(2):291--330, 1982.

\bibitem{Elliott96}
George~A. Elliott, Toshikazu Natsume, and Ryszard Nest.
\newblock The {A}tiyah-{S}inger index theorem as passage to the classical limit
  in quantum mechanics.
\newblock {\em Comm. Math. Phys.}, 182(3):505--533, 1996.

\bibitem{Green78}
Philip Green.
\newblock The local structure of twisted covariance algebras.
\newblock {\em Acta Math.}, 140(3-4):191--250, 1978.

\bibitem{Gromov83}
Mikhael Gromov and H.~Blaine Lawson, Jr.
\newblock Positive scalar curvature and the {D}irac operator on complete
  {R}iemannian manifolds.
\newblock {\em Inst. Hautes \'Etudes Sci. Publ. Math.}, (58):83--196 (1984),
  1983.

\bibitem{Higson93}
Nigel Higson and John Roe.
\newblock On the coarse {B}aum--{C}onnes conjecture.
\newblock In {\em Novikov conjectures, index theorems and rigidity, {V}ol.\ 2
  ({O}berwolfach, 1993)}, volume 227 of {\em London Math. Soc. Lecture Note
  Ser.}, pages 227--254. Cambridge Univ. Press, Cambridge, 1995.

\bibitem{higson00}
Nigel Higson and John Roe.
\newblock {\em Analytic {$K$}-homology}.
\newblock Oxford Mathematical Monographs. Oxford University Press, Oxford,
  2000.
\newblock Oxford Science Publications.

\bibitem{Mathai13}
Peter Hochs and Varghese Mathai.
\newblock Geometric quantization and families of inner products.
\newblock {\em Adv. Math.}, 282:362--426, 2015.

\bibitem{HochsSong16b}
Peter Hochs and Yanli Song.
\newblock An equivariant index for proper actions {II}: properties and
  applications.
\newblock ArXiv:1602.02836, 2016.

\bibitem{HochsSongDS}
Peter Hochs and Yanli Song.
\newblock An equivariant index for proper actions {III}: the invariant and
  discrete series indices.
\newblock {\em Differential Geom. Appl.}, 49:1--22, 2016.

\bibitem{HochsSong15}
Peter Hochs and Yanli Song.
\newblock Equivariant indices of {S}pin$^c$-{D}irac operators for proper moment
  maps.
\newblock {\em Duke Math. J.}, to appear, 2016.
\newblock ArXiv:1503.00801.

\bibitem{Kasparov83}
Gennadi Kasparov.
\newblock Index of invariant elliptic operators, {$K$}-theory and
  representations of {L}ie groups.
\newblock {\em Dokl. Akad. Nauk SSSR}, 268(3):533--537, 1983.

\bibitem{Kasparov14}
Gennadi Kasparov.
\newblock Elliptic and transversally elliptic index theory from the viewpoint
  of {$KK$}-theory.
\newblock {\em J. Noncommut. Geom.}, to appear, 2016.
\newblock Preprint version November 2013.

\bibitem{Kucerovsky01}
Dan Kucerovsky.
\newblock A short proof of an index theorem.
\newblock {\em Proc. Amer. Math. Soc.}, 129(12):3729--3736, 2001.

\bibitem{Zhang14}
Xiaonan Ma and Weiping Zhang.
\newblock Geometric quantization for proper moment maps: the {V}ergne
  conjecture.
\newblock {\em Acta Math.}, 212(1):11--57, 2014.

\bibitem{Mathai10}
Varghese Mathai and Weiping Zhang.
\newblock Geometric quantization for proper actions.
\newblock {\em Adv. Math.}, 225(3):1224--1247, 2010.
\newblock With an appendix by Ulrich Bunke.

\bibitem{Paradan11}
Paul-{\'E}mile Paradan.
\newblock Formal geometric quantization {II}.
\newblock {\em Pacific J. Math.}, 253(1):169--211, 2011.

\bibitem{Pflaum15}
Markus Pflaum, Hessel Posthuma, and Xiang Tang.
\newblock The transverse index theorem for proper cocompact actions of {L}ie
  groupoids.
\newblock {\em J. Differential Geom.}, 99(3):443--472, 2015.

\bibitem{Rieffel82}
Marc Rieffel.
\newblock Applications of strong {M}orita equivalence to transformation group
  {$C^{\ast} $}-algebras.
\newblock In {\em Operator algebras and applications, {P}art {I} ({K}ingston,
  {O}nt., 1980)}, volume~38 of {\em Proc. Sympos. Pure Math.}, pages 299--310.
  American Mathematical Society, Providence, RI, 1982.

\bibitem{Zhang98}
Youliang Tian and Weiping Zhang.
\newblock An analytic proof of the geometric quantization conjecture of
  {G}uillemin-{S}ternberg.
\newblock {\em Invent. Math.}, 132(2):229--259, 1998.

\bibitem{Wang14}
Hang Wang.
\newblock {$L^2$}-index formula for proper cocompact group actions.
\newblock {\em J. Noncommut. Geom.}, 8(2):393--432, 2014.

\bibitem{Williams07}
Dana Williams.
\newblock {\em Crossed products of {$C{^\ast}$}-algebras}, volume 134 of {\em
  Mathematical Surveys and Monographs}.
\newblock American Mathematical Society, Providence, RI, 2007.

\end{thebibliography}

\end{small}

\end{document}